\documentclass[11pt,final,reqno]{article}
\usepackage{fullpage}
\usepackage{amsmath, amsfonts,amssymb,amsthm}
\usepackage{graphicx}
\usepackage{color}
\usepackage{verbatim}
\usepackage{enumitem}
\usepackage{hyperref}
\usepackage[sharp]{easylist}
\usepackage[labelformat=empty]{subfig}

\newtheorem{theorem}{Theorem}[section]
\newtheorem{lemma}[theorem]{Lemma}
\newtheorem{proposition}[theorem]{Proposition}

\newtheorem{definition}[theorem]{Definition}
\newtheorem{conjecture}{Conjecture}
\newtheorem{claim}[theorem]{Claim}
\newtheorem{fact}[theorem]{Fact}

\newcommand{\ep}{\epsilon}

\newcommand{\floor}[1]{\left\lfloor#1\right\rfloor}

\newcommand{\croot}[1]{\left\lceil\sqrt{#1}\right\rceil}

\title{Tiling in bipartite graphs with asymmetric minimum degrees}
\author{
Andrzej Czygrinow\thanks{School of Mathematical and Statistical Sciences, Arizona State University, Tempe, AZ 85287, USA.  E-mail address: andrzej.czygrinow@asu.edu.  Research of this author is supported in part by NSA grant H98230-08-1-0046}
\quad and \quad
Louis DeBiasio\thanks{Department of Mathematics, Miami University, Oxford, OH 45056 USA. E-mail address: debiasld@miamioh.edu.}
}

\begin{document}
\maketitle

\begin{abstract}
The problem of determining the optimal minimum degree condition for a balanced bipartite graph on $2ms$ vertices to contain $m$ vertex disjoint copies of $K_{s,s}$ was solved by Zhao \cite{Z}.  Later Hladk\'y and Schacht \cite{HS}, and Czygrinow and DeBiasio \cite{CD} determined the optimal minimum degree condition for a balanced bipartite graph on $2m(s+t)$ vertices to contain $m$ vertex disjoint copies of $K_{s,t}$ for fixed positive integers $s<t$.  

For a balanced bipartite graph $G[U,V]$, let $\delta_U=\min\{\deg(u):u\in U\}$ and $\delta_V=\min\{\deg(v): v\in V\}$.  We consider the problem of determining the optimal value of $\delta_U+\delta_V$ which guarantees that $G$ can be tiled with $K_{s,s}$.  We show that the optimal value depends on $D:=|\delta_V-\delta_U|$.  When $D$ is small, we show that $\delta_U+\delta_V\geq n+3s-5$ is best possible.  As $D$ becomes larger, we show that $\delta_U+\delta_V$ can be made smaller, but no smaller than $n+2s-2\croot{s}$.  However, when $D=n-C$ for some constant $C$, we show that there exist graphs with $\delta_U+\delta_V\geq n+s^{s^{1/3}}$ which cannot be tiled with $K_{s,s}$.
\end{abstract}

\section{Introduction}

If $G$ is a graph on $n=sm$ vertices, $H$ is a graph on $s$ vertices and $G$ contains $m$ vertex disjoint copies of $H$, then we say $G$ can be \emph{tiled} with $H$.  We now state two important tiling results which motivate the current research.

\begin{theorem}[Hajnal-Szemer\'edi \cite{HSz}]\label{hsz}
Let $G$ be a graph on $n=sm$ vertices.  If $\delta(G)\geq (s-1)m$, then $G$ can be tiled with $K_s$.
\end{theorem}

Kierstead and Kostochka generalized, and in doing so slightly improved, the result of Hajnal and Szemer\'edi.

\begin{theorem}[Kierstead-Kostochka \cite{KK}]
Let $G$ be a graph on $n=sm$ vertices.  If $\deg(x)+\deg(y)\geq 2(s-1)m-1$, for all non-adjacent $x,y\in V(G)$ then $G$ can be tiled with $K_s$.
\end{theorem}

Both of these results can be shown to be best possible relative to the respective degree condition, i.e. no smaller lower bound on the degree will suffice.

For the rest of the paper we will consider tiling in bipartite graphs.  Given a bipartite graph $G[U,V]$ we say $G$ is balanced if $|U|=|V|$.  The following theorem is a consequence of Hall's matching theorem, and is an early result on bipartite graph tiling.

\begin{theorem}\label{Hall1}
Let $G$ be a balanced bipartite graph on $2n$ vertices.  If $\delta(G)\geq \frac{n}{2}$, then $G$ can be tiled with $K_{1,1}$.
\end{theorem}

Zhao determined the best possible minimum degree condition for a bipartite graph to be tiled with $K_{s,s}$ when $s\geq 2$.

\begin{theorem}[Zhao \cite{Z}]\label{Zhao theorem}
For each $s\geq 2$, there exists $m_0$ such that the following holds for all $m\geq m_0$.  If $G$ is a balanced bipartite graph on $2n=2ms$ vertices with 
$$\delta(G)\geq \left\lbrace \begin{array}{ll} \frac{n}{2}+s-1   & \text{ if } m \text{ is even } \\
              \frac {n+3s}{2}-2 & \text{ if } m \text{ is odd, } \end{array} \right. $$
then $G$ can be tiled with $K_{s,s}$.
\end{theorem}

Hladk\'y and Schacht, and the authors determined the best possible minimum degree condition for a balanced bipartite graph to be tiled with $K_{s,t}$.

\begin{theorem}[Hladk\'y, Schacht \cite{HS}; Czygrinow, DeBiasio \cite{CD}]
For each $t>s\geq 1$, there exists $m_0$ such that the following holds for all $m\geq m_0$.  If $G$ is a balanced bipartite graph on $2n=2m(s+t)$ vertices with 
$$\delta(G)\geq \left\lbrace \begin{array}{ll} \frac{n}{2}+s-1   & \text{ if } m \text{ is even } \\
               \frac {n+t+s}{2}-1 & \text{ if } m \text{ is odd and } t\leq 2s \\
               \frac {n+3s}{2}-1 & \text{ if } m \text{ is odd and } t\geq 2s+1  \end{array} \right. $$
then $G$ can be tiled with $K_{s,s}$.
\end{theorem}

Now we consider a more general degree condition than $\delta(G)$.  Given a bipartite graph $G[U,V]$, let $\delta_U(G):=\min\{\deg(u): u\in U\}$ and $\delta_V(G):=\min\{\deg(v): v\in V\}$.  We will write $\delta_U$ and $\delta_V$ instead of $\delta_U(G)$ and $\delta_V(G)$ when it is clear which graph we are referring to.  The following theorem is again a consequence of Hall's matching theorem and is more general than Theorem \ref{Hall1}.

\begin{theorem}\label{Hall2}
Let $G[U,V]$ be a balanced bipartite graph on $2n$ vertices.  If $\delta_U+\delta_V\geq n$, then $G$ can be tiled with $K_{1,1}$.
\end{theorem}

Notice that when $s=2$, Theorem \ref{Zhao theorem} says that if $G[U,V]$ is a balanced bipartite graph on $2n$ vertices with $\delta(G)\geq \frac{n}{2}+1$, then $G$ can be tiled with $K_{2,2}$.  
Wang made the following general conjecture about $2$-factors in bipartite graphs which would in particular provide an analog of Theorem \ref{Hall2} for tiling with $K_{2,2}$.

\begin{conjecture}[Wang \cite{W2}]
\label{con:W2}
Let $G[U,V]$ and $H$ be balanced bipartite graphs on $2n$ vertices.  If $\delta_U+\delta_V\geq n+2$ and $\Delta(H)\leq 2$, then $H\subseteq G$.
\end{conjecture}

The authors together with Kierstead \cite{CDK} proved Wang's conjecture when $\delta_V\geq \delta_U=\Omega(n)$ and $n$ is sufficiently large. 

The purpose of this paper is to explore a generalization of Theorem \ref{Zhao theorem} in the way that Theorem \ref{Hall2} generalizes Theorem \ref{Hall1}.  As we will see, this generalization turns out to be less straightforward than one might anticipate.  Our first result is as follows.

\begin{theorem}\label{main 1}
For all $s\geq 2$ and $\lambda\in (0,\frac{1}{2})$, there exists $m_0$ such that the following holds for all $m\geq m_0$.  If $G[U,V]$ is a balanced bipartite graph on $2n=2ms$ vertices with $\delta_V\geq\delta_U\geq \lambda n$ and $\delta_U+\delta_V\geq n+3s-5,$
then $G$ can be tiled with $K_{s,s}$.
\end{theorem}

Note that a specific instance of Theorem \ref{main 1} is that for sufficiently large $n$ and $\delta_V\geq \delta_U=\Omega(n)$, $\delta_U+\delta_V\geq n+1$ is sufficient for tiling with $K_{2,2}$ (compare this statement to Conjecture \ref{con:W2}).

Perhaps surprisingly, we show that a smaller degree sum will suffice when the difference between $\delta_V$ and $\delta_U$ is large enough.  In order to precisely state our second result we need the following definition.  

\begin{definition}
Let $c:\mathbb{Z}^+\to \{0,1\}$ such that 
$$c(s) = \left\lbrace \begin{array}{ll} 0   & \text{ if } q=0 ~\text{ or }~ p+1\leq q\leq 2p\\
               1 & \text{ if } 1\leq q\leq p  \end{array} \right. $$
where $p$ and $q$ are the unique non-negative integers satisfying $s=p^2+q$ and $0\leq q\leq 2p$.
\end{definition} 

\begin{theorem}\label{main 2}
For all $s\geq 2$ and $\lambda\in (0,\frac{1}{2})$, there exists $m_0$ such that the following holds for all $m\geq m_0$.  Let $G[U,V]$ be a balanced bipartite graph on $2n=2ms$ vertices with $\delta_V\geq\delta_U\geq \lambda n$ and let $k_1$ and $k_2$ be the unique integers such that $k_1+k_2=m$ and $\delta_U=k_1s+s+r$ with $0\leq r\leq s-1$.  For all $0\leq d\leq s-2\croot{s}+c(s)+1$, if 
$k_2\geq (s-d)k_1$ and $$\delta_U+\delta_V\geq n+2s-2\croot{s}+d+c(s),$$
then $G$ can be tiled with $K_{s,s}$.
\end{theorem}

%


As mentioned earlier, Zhao gave examples which show that Theorem \ref{Zhao theorem} is best possible.  In particular, \cite{Z} contains an example of a bipartite graph $G_0$ with $\delta(G_0)=\frac{n+3s}{2}-3$ which cannot be tiled with $K_{s,s}$.  Consequently, 
%
%
there are examples with $\delta_U+\delta_V=2\delta(G)=n+3s-6$ which cannot be tiled with $K_{s,s}$.  So the degree condition in Theorem \ref{main 1} cannot be improved in general.  Notice that Theorem \ref{Zhao theorem} gives a better bound on $\delta(G)$ when $m$ is even, which may seem to suggest that $\delta_U+\delta_V\geq n+2s-3$ suffices when $m$ is even (based on Theorem \ref{main 1}).  However, we show that when $m$ is even (or odd) there are graphs with $\delta_U+\delta_V=n+3s-7$ that cannot be tiled with $K_{s,s}$.

\begin{proposition}\label{meven}
Let $s\geq 2$.  For every $j\in \mathbb{N}$, there exists an integer $m$ and a balanced bipartite graph $G[U,V]$ on $2n=2ms$ vertices such that $\delta_U+\delta_V=n+3s-7$ and $2sj-s-1\leq |\delta_V-\delta_U|\leq 2sj-1$, but $G$ cannot be tiled with $K_{s,s}$.
\end{proposition}

We also give examples to show that the degree is tight when $d=0$ in Theorem \ref{main 2}.

\begin{proposition}\label{k1smallexample}
For every $s\geq 2$, there exists a balanced bipartite graph $G[U,V]$ with $k_2\geq sk_1$ and
$$\delta_U+\delta_V=n+2s-2\croot{s}+c(s)-1$$ such that $G$ cannot be tiled with $K_{s,s}$.  
\end{proposition}

%
%
%

Finally, when $\delta_U$ is constant, 
we show that there exist graphs (without constructing them) with $\delta_U+\delta_V$ much larger than $n+3s$ which cannot be tiled with $K_{s,s}$.

\begin{proposition}\label{probexample}
There exists $s_0, n_0\in \mathbb{N}$ such that for all $s\geq s_0$, there exists a graph $G[U,V]$ on $n\geq n_0$ vertices with $\delta_U+\delta_V=n+s^{s^{1/3}}$ such that $G$ cannot be tiled with $K_{s,s}$.
\end{proposition}

The following figure summarizes the results of Theorems \ref{main 1} and \ref{main 2} and Propositions \ref{meven}, \ref{k1smallexample}, and \ref{probexample} by plotting the degree sum needed for tiling with $K_{s,s}$ in terms of the difference between $\delta_V$ and $\delta_U$.  The first grey area in the figure represents a range of values of $\delta_V-\delta_U$ for which we cannot provide a matching lower bound on $\delta_U+\delta_V$.  The second grey area represents a range of values of $\delta_V-\delta_U$ for which we cannot provide non-trivial upper or lower bounds on $\delta_U+\delta_V$.

\begin{figure}[ht]
\centering
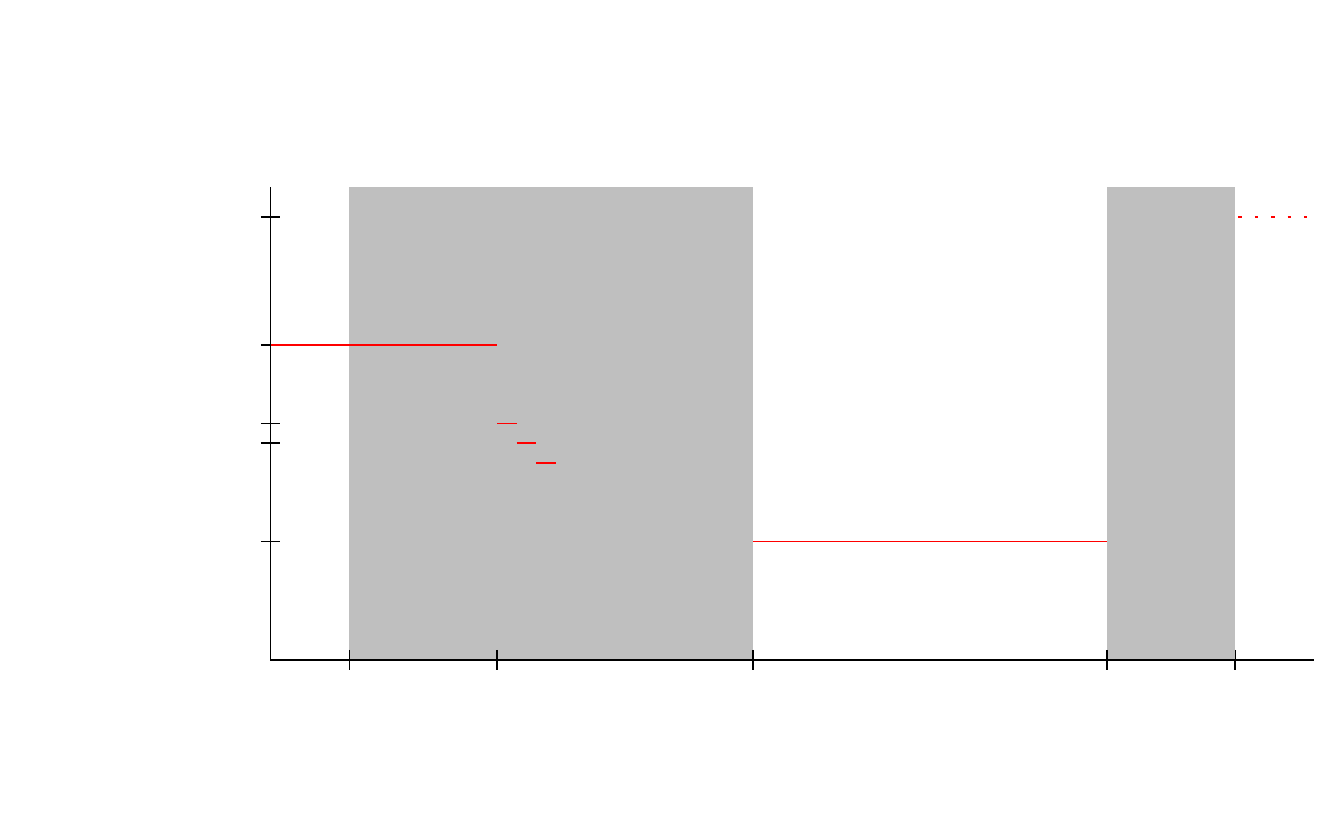
\label{regimes}
\caption{}
\end{figure}

\section{Extremal Examples}

\subsection{Tightness when $\delta_V-\delta_U$ is constant}

As mentioned in the introduction, Zhao determined the optimal minimum degree condition so that $G$ can be tiled with $K_{s,s}$.  If $n$ is an odd multiple of $s$, then $\delta(G)\geq \frac{n}{2}+\frac{3s}{2}-2$ is best possible; however, if $n$ is an even multiple of $s$, then $\delta(G)\geq \frac{n}{2}+s-1$ is best possible.  In Theorem \ref{main 1} and Theorem \ref{main 2} we show that if $\delta_V\geq \delta_U=\Omega(n)$, then $\delta_U+\delta_V\geq n+3s-5$ suffices to give a tiling of $G$ with $K_{s,s}$.  We now give an example which shows that even when $n$ is an even multiple of $s$, we cannot improve the coefficient of the $s$ term in the degree condition.

We will need to use the graphs $P(m,p)$, where $m,p\in\mathbb{N}$, introduced by Zhao in \cite{Z}.

\begin{lemma}\label{No K22}

For all $p\in\mathbb{N}$ there exists $m_0$ such that for all $m\in \mathbb{N}$, $m>m_0$, there exists a balanced bipartite graph, $P(m,p)$, on $2m$ vertices, so that the following hold:
\begin{enumerate}
\item $P(m,p)$ is $p$-regular

\item $P(m,p)$ does not contain a copy of $K_{2,2}$.
\end{enumerate}

\end{lemma}

First we recall Zhao's example which shows that there exist graphs with $\delta_U+\delta_V=n+3s-6$ such that $G$ cannot be tiled with $K_{s,s}$.  Let $G[U, V]$ be a balanced bipartite graph on $2n$ vertices with $n=(2k+1)s$.  Partition $U$ as $U_1\cup U_2$ with $|U_1|=ks+1$, $|U_2|=ks+s-1$ and partition $V$ as $V_1\cup V_2$ with $|V_1|=ks+s-1$, $|V_2|=ks+1$.  Let $G[U_1, V_1]$ and $G[U_2, V_2]$ be complete, let $G[U_1, V_2]\simeq P(ks+1, s-2)$ and let $G[U_2, V_1]\simeq P(ks+s-1, 2s-4)$.

\begin{figure}[ht]
\centering
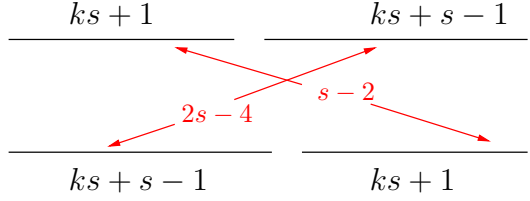
\label{3s-6zhao}
\caption{$m$ is odd and $\delta_U+\delta_V=n+3s-6$}
\end{figure}

We now recall the argument which shows that $G$ cannot be tiled with $K_{s,s}$.  Suppose $G$ can be tiled with $K_{s,s}$ and let $\mathcal{K}$ be such a tiling.  For $F\in \mathcal{K}$ and $i=1,2$, let $X_i(F):=V(F)\cap U_i$, $Y_i(F):=V(F)\cap V_i$ and $\vec{v}(F)=(|X_1(F)|, |X_2(F)|, |Y_1(F)|, |Y_2(F)|)$.  We say $F\in \mathcal{K}$ is \emph{crossing} if $V(F)\cap (U_1\cup V_1)\neq \emptyset$ and $V(F)\cap (U_2\cup V_2)\neq \emptyset$.  We now claim that if $F$ is crossing then $\vec{v}(F)=(s-1,1,s,0)$ or $\vec{v}(F)=(0,s,1,s-1)$.  It is not possible for $X_1(F)\neq \emptyset$ and $Y_2(F)\neq \emptyset$ since $G[U_1, V_2]\simeq P(ks+1, s-2)$ and $G[V_1, U_2]$ is $K_{2,2}$-free.  Thus if $X_1(F)\neq \emptyset$, then $|Y_1(F)|=s$, $|X_2(F)|\leq 1$, and $|X_1(F)|\geq s-1$.  If $Y_2(F)\neq \emptyset$, then $|X_2(F)|=s$, $|Y_1(F)|\leq 1$, and $|Y_2(F)|\geq s-1$.  This shows that if $F$ is crossing then $\vec{v}(F)=(s-1,1,s,0)$ or $\vec{v}(F)=(0,s,1,s-1)$.  Finally, since we are supposing that $G$ can be tiled, there exists some $\ell\in \mathbb{N}$ and some subset $\mathcal{K}'\subseteq \mathcal{K}$ such that every $F\in \mathcal{K}'$ is crossing and $\sum_{F\in \mathcal{K}'}|X_1(F)|=\ell s+1$ and $\sum_{F\in \mathcal{K}'}|Y_1(F)|=\ell s+s-1$.  Let $i_1$ be the number of $F\in \mathcal{K}'$ with $\vec{v}(F)=(s-1,1,s,0)$ and let $i_2$ be the number of $F\in \mathcal{K}'$ with $\vec{v}(F)=(0,s,1,s-1)$.  Then we have
\begin{align*}
\text{(i)} ~~(s-1)i_1=\ell s+1 ~~~\text{ and }~~~ \text{(ii)} ~~si_1+i_2=\ell s+s-1
\end{align*}
Which implies $i_1+i_2=s-2$.  However, (ii) implies that $i_2\geq s-1$, a contradiction.

Now we prove Theorem \ref{meven}.

\begin{proof}
 
%
%

We give two examples of graphs which cannot be tiled with $K_{s,s}$; one when $m$ is even, one $m$ is odd, and both with $\delta_U+\delta_V=n+3s-7$.  

Let $j$ be a non-negative integer and let $m=2k$, where $k$ is sufficiently large.  Let $U$ and $V$ be sets of vertices such that $|U|=|V|=2ks$.  Let $U$ be partitioned as $U=U_1\cup U_2$ and $V$ be partitioned as $V=V_1\cup V_2$ with $|U_1|=(k-j)s+1$, $|U_2|=(k+j)s-1$, $|V_1|=(k-j+1)s-1$ and $|V_2|=(k+j-1)s+1$.  Let $G[U_i, V_i]$ be complete for $i=1,2$.  Let $G[U_1, V_2]$ be the graph obtained from $G[U_1', V_2]\simeq P((k+j)s-s+1, s-2)$ by deleting $(2j-1)s$ vertices from $U_1'$ while maintaining $\delta(V_2, U_1)\geq s-3$ (note that when $s=2$, $\delta(V_2, U_1)=0$).  Let $G[U_2, V_1]$ be the graph obtained from $G[U_2, V_1']\simeq P((k+j)s-1, (2j+1)s-5)$ by deleting $(2j-1)s$ vertices from $V_1'$ while maintaining $\delta(U_2, V_1)\geq (2j+1)s-6$.
We have 
\begin{align*}
\delta_U&=(k-j)s+s-1+s-2=(k-j+2)s-3,\\
\delta_V&=(k+j)s-1+s-3=(k-j)s+1+(2j+1)s-5=(k+j+1)s-4,
\end{align*}
and thus $\delta_U+\delta_V= 2ks+3s-7= n+3s-7.$

\begin{figure}[ht]
\centering
\scalebox{.85}{\subfloat[Case: $m$ even]{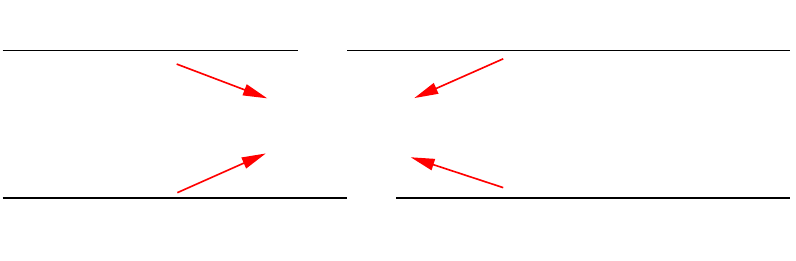
\label{3s-7even}}
}~~~~~
\scalebox{.85}{\subfloat[Case: $m$ odd]{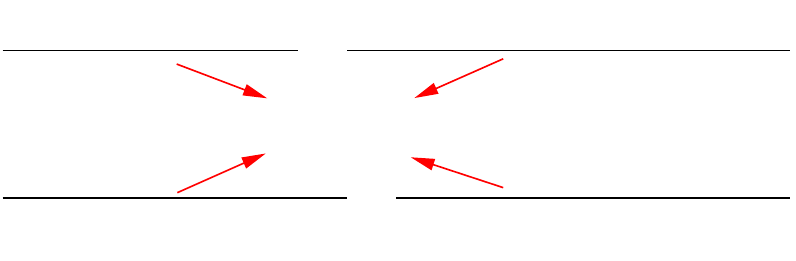
\label{3s-7odd}}
}
\label{3s-7}
\caption[$\delta_U+\delta_V=n+3s-7$]{$\delta_U+\delta_V=n+3s-7$}
\end{figure}

Let $j$ be a non-negative integer and let $m=2k+1$, where $k$ is sufficiently large.  Let $U$ and $V$ be sets of vertices such that $|U|=|V|=(2k+1)s$.  Let $U$ be partitioned as $U=U_1\cup U_2$ and $V$ be partitioned as $V=V_1\cup V_2$ with $|U_1|=(k-j)s+1$, $|U_2|=(k+j)s+s-1$, $|V_1|=(k-j)s+s-1$ and $|V_2|=(k+j)s+1$.  Let $G[U_i, V_i]$ be complete for $i=1,2$.  Let $G[U_1, V_2]$ be the graph obtained from $G[U_1', V_2]\simeq P((k+j)s+1, s-2)$ by deleting $2js$ vertices from $U_1'$ while maintaining $\delta(V_2, U_1)\geq s-3$ (note that when $s=2$, $\delta(V_2, U_1)=0$).  Let $G[U_2, V_1]$ be the graph obtained from $G[U_2, V_1']\simeq P((k+j)s+s-1, (2j+2)s-5)$ by deleting $2js$ vertices from $V_1'$ while maintaining $\delta(U_2, V_1)\geq (2j+2)s-6$.
We have 
\begin{align*}
\delta_U&=(k-j)s+s-1+s-2=(k-j+2)s-3,\\
\delta_V&=(k+j)s+s-1+s-3=(k-j)s+1+(2j+2)s-5=(k+j+2)s-4,
\end{align*}
and thus $\delta_U+\delta_V= (2k+1)s+3s-7= n+3s-7.$

The same analysis given before the start of this proof shows that each of these graphs cannot be tiled with $K_{s,s}$.

\end{proof}

\subsection{Tightness when $\delta_V-\delta_U$ is large}

%
%
%
%
%
%
%
%
%
%
%

Now we prove Theorem \ref{k1smallexample}.

\begin{proof}


\begin{figure}[ht]
\centering
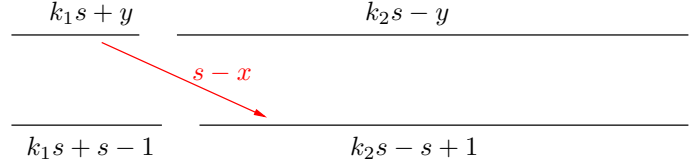
\caption[$\delta_U+\delta_V=n+2s-x-y-1$]{$\delta_U+\delta_V=n+2s-x-y-1$}
\end{figure}

Let $G=(U_1\cup U_2, V_1\cup V_2; E)$ be a bipartite graph with $|U_1|=k_1s+y$, $|U_2|=k_2s-y$, $|V_1|=k_1s+s-1$, $|V_2|=k_2s-s+1$ such that $G[U_1,V_1]$, $G[U_2,V_2]$, and $G[V_1,U_2]$ are complete.  Furthermore suppose $|V_2|\geq (s-x)|U_1|$, every vertex in $U_1$ has $s-x$ neighbors in $V_2$, and for all $u, u'\in U_1$, $(N(u)\cap N(u'))\cap V_2=\emptyset$.  Thus we have $0\leq \delta(V_2,U_1)\leq\Delta(V_2,U_1)\leq 1$ with $\delta(V_2, U_1)=\Delta(V_2, U_1)=1$ only when $|V_2|=(s-x)|U_1|$ and thus
\begin{equation}
\delta_U+\delta_V\geq k_1s+s-1+s-x+k_2s-y=n+2s-(x+y)-1 \label{x+y}
\end{equation}

Every copy of $K_{s,s}$ which touches both $U_1$ and $U_2\cup V_2$ must have one vertex from $U_1$, $s-1$ vertices from $U_2$, at most $s-x$ vertices from $V_2$, and at least $x$ vertices from $V_1$.  So if $xy\geq s$, then $G$ cannot be tiled.   So in order to maximize $\delta_U+\delta_V$ we minimize $x+y$ subject to the condition that $xy\geq s$.  The result is that $x=y=\croot{s}$, unless $1\leq q\leq p$ in which case $x=\croot{s}-1$, $y=\croot{s}$ suffices.  Thus \eqref{x+y} gives $\delta_U+\delta_V= n+2s-2\croot{s}-1$ in general and $\delta_U+\delta_V=n+2s-2\croot{s}$ when $1\leq q\leq p$.

\end{proof}

\section{Non-extremal Case}

In order to prove Theorem \ref{main 1} and Theorem \ref{main 2} we will first prove the following Theorem.  

\begin{theorem}\label{non-extreme}
For every $\alpha>0$ and every positive integer $s$, there exist $\beta>0$ and positive integer $m_1$ such that the following holds for all $n=ms$ with $m\geq m_1$.  Given a bipartite graph $G[U,V]$ with $|U|=|V|=n$, if $\delta_U+\delta_V\geq (1-2\beta)n$, $\delta_V\geq \delta_U\gg \alpha n$ and $\delta_U=k_1s+s+r$ for some $0\leq r\leq s-1$ with $k_1+k_2=m$, then either $G$ can be tiled with $K_{s,s}$, or 
\begin{equation}\label{extremalcondition}
\text{there exist } U_1'\subseteq U,~ V_2'\subseteq V, \text{ such that } |U_1'|=k_1s,~ |V_2'|=k_2s,~ d(U_1',V_2')\leq \alpha.
\end{equation}
\end{theorem}

If $G$ is a graph for which \eqref{extremalcondition} holds, then we say $G$ satisfies the \emph{extremal condition with parameter $\alpha$}.

\subsection{Regularity and Blow-Up Lemmas}

In this section we review the Regularity and Blow-up Lemmas. Let $\varGamma$ be a simple graph on $n$ vertices. For two disjoint, nonempty subsets $U$ and $V$ of $V(\varGamma)$, define the density of the pair $(U,V)$ as
\[
d(U,V)=\frac{e(U,V)}{|U||V|}.
\]

\begin{definition}
A pair $(U,V)$ is called $\ep$-\emph{regular} if for every $%
U^{\prime }\subseteq U$ with $|U^{\prime }|\geq \ep |U|$ and every $%
V^{\prime }\subseteq V$ with $|V^{\prime }|\geq \ep |V|$, $
|d(U^{\prime },V^{\prime })-d(U,V)|\leq \ep $. The pair $\left(
U,V\right) $ is $( \ep ,\delta ) $-\emph{super-regular}
if it is $\ep $-regular and for all $u\in U$, 
$\deg\left(
u,V\right) \geq \delta \left| V\right| $ and for all $v\in V$, $\deg\left(
v,U\right) \geq \delta \left| U\right| $.
\end{definition}

First we note the following facts that we will need.

\begin{fact}[Intersection Property]\label{interprop}
\label{deg} If $(U,V)$ is an $\ep $-regular pair with density $d$,
then for any $Y\subseteq V$ with $(d-\ep)^{k-1}|Y|\geq \ep |V|$ there are less than
$k\ep |U|^k$ $k$-tuples of vertices $(u_1, u_2, \dots, u_k)$, $u_i\in U$, such that $|Y\cap N(u_1, u_2, \dots, u_k)|\leq (d-\ep)^k|Y|$.
\end{fact}

\begin{fact}[Slicing Lemma]\label{slicing}
\label{slice} Let $(U,V)$ be an $\ep $-regular pair with density $d$, and for some $\lambda >\ep $ let $U^{\prime }\subseteq U$, $V^{\prime }\subseteq V$, with $|U^{\prime }|\geq \lambda |U|$, $|V^{\prime }|
\geq \lambda |V|$. Then $(U^{\prime },V^{\prime })$ is an $\ep
^{\prime }$-regular pair of density $d^{\prime }$ where $\ep
^{\prime }=\max \{\frac{\ep }{\lambda },2\ep \}$ and $
d^{\prime}\ge d-\ep $.
\end{fact}

Our main tool in the proof will be the Regularity Lemma of Szemer\'{e}di \cite{Sz} which we state in its multipartite form.

\begin{lemma}[Regularity Lemma - Bipartite Version]\label{bireg} 
For every $\ep>0$ there exists $M:=M(\ep)$ such that if $G:=G[U,V]$ is a balanced bipartite graph on $2n$ vertices and $d\in[0,1]$, then there is a partition of $U$ into clusters $U_0, U_1,\dots, U_t$, a partition of $V$ into clusters $V_0, V_1,\dots, V_t$, and a subgraph $G':=G'[U,V]$ with the following properties: 

\begin{enumerate}
\item $t\leq M$,
\item $|U_0|\leq \ep n$, $|V_0|\leq \ep n$,
\item $|U_i|=|V_i|=\ell\leq \ep n$ for all $i\in [t]$,
\item $\deg_{G'}(x)>\deg_G(x)-(d+\ep)n$ for all $x\in V(G)$,
\item All pairs $(U_i, V_i)$, $i,j\in [t]$, are $\ep$-regular in $G'$ each with density either $0$ or exceeding $d$.

\end{enumerate}
\end{lemma}

In addition, we will use the Blow-up Lemma of Koml\'{o}s, S\'{a}rk\"{o}zy, and Szemer\'{e}di \cite{KSSbu}.

\begin{lemma}[Blow-up Lemma]\label{blowup}
Given $\delta >0$, $\Delta >0$ there exists $\ep >0$ such that the following holds. Let $(U, V)$ be an $(\ep ,\delta )$-super-regular pair. 
If $T$ is a $U',V'$-bigraph with maximum degree $\Delta (T)\leq\Delta $ and $T$ is embeddable into the complete bipartite graph $K_{|U|},{|V|}$ then it is also embeddable into $(U,V)$.
\end{lemma}

\subsection{Proof of Theorem \ref{non-extreme}}

Here we prove Theorem \ref{non-extreme}.  We show that if $G$ is not in the extremal case, we obtain a tiling with $K_{s,s}$; otherwise $G$ is in the extremal case which we deal with in Section \ref{extremesection}.  The proof is adopted from Zhao \cite{Z}.

\begin{proof}Let $\ep$, $d$, and $\beta$ be positive real numbers such that $$\ep\ll d\ll \beta\ll \alpha$$ and suppose $n$ is large.  Let $G[U,V]$ be a bipartite graph with $|U|=|V|=n$,  $\delta_U+\delta_V\geq (1-\beta)n$, and $\delta_V\geq \delta_U\gg \alpha n$.  We also have $\delta_U=k_1s+s+r$ for some $0\leq r\leq s-1$ and we set $k_2:=m-k_1$.  Let $\gamma_1, \gamma_2$ be positive real numbers such that $\delta_U\geq (\gamma_1-\beta)n$, $\delta_V\geq (\gamma_2-\beta) n$ and $\gamma_1+\gamma_2=1$.  Note that $\gamma_2\geq \gamma_1\gg \alpha$.  We apply Lemma \ref{bireg} to $G$ with parameters $\ep$ and $d$.  We obtain a partition of $U$ into $U_0, U_1,\dots, U_t$ and $V$ into $V_0, V_1,\dots, V_t$ such that $|U_i|=|V_i|=\ell\leq \ep n$ for all $i\in [t]$ and $|U_0|=|V_0|\leq \ep n$.  In the graph $G'$ from Lemma \ref{bireg}, we have $(U_i, V_j)$, is $\ep$-regular with density either $0$ or exceeding $d$ for all $i,j\in [t]$.  We also have $\deg_{G'}(u)>(\gamma_1-\beta)n-(\ep+d)n$ for $u\in U$ and $\deg_{G'}(v)>(\gamma_2-\beta)n-(\ep+d)n$ for $v\in V$.

We now consider the \emph{reduced graph} of $G'$.  Let $G_r$ be a bipartite graph with parts $\mathcal{U}:=\{U_1, \dots, U_t\}$ and $\mathcal{V}:=\{V_1,\dots, V_t\}$ such that $U_i$ is adjacent to $V_j$, denoted $U_i\sim V_j$, if and only if $(U_i, V_j)$ is an $\ep$-regular pair with density exceeding $d$.  A standard calculation gives the following degree condition in the reduced graph, $\delta_\mathcal{U}\geq (\gamma_1-2\beta)t$ and $\delta_\mathcal{V}\geq (\gamma_2-2\beta)t$.

\begin{claim}\label{split}
If $G_r$ contains two subsets $X\subseteq \mathcal{U}$ and $Y\subseteq \mathcal{V}$ such that $|X|\geq (\gamma_1-3\beta)t$, $|Y|\geq (\gamma_2-3\beta)t$ and there are no edges between $X$ and $Y$, then \eqref{extremalcondition} holds in $G$.
\end{claim}

\begin{proof}
Without loss of generality, assume that $|X|=(\gamma_1-3\beta)t$ and $|Y|=(\gamma_2-3\beta)t$.  Let $U'=\cup_{U_i\in X}U_i$ and $V'=\cup_{V_i\in Y}V_i$.  We have 
$$(\gamma_1-4\beta)n<(\gamma_1-3\beta)t\ell=|X|\ell=|U'|\leq (\gamma_1-3\beta)n$$
and
$$(\gamma_2-4\beta)n<(\gamma_2-3\beta)t\ell=|Y|\ell=|V'|\leq (\gamma_2-3\beta)n.$$
Since there is no edge between $X$ and $Y$ we have $e_{G'}(U',V')=0$.  Consequently $e_G(U', V')\leq e_{G'}(U', V')+d|U'||V'|+2\ep n|U'|<dk_1sk_2s$.  By adding at most $4\beta k_1s$ vertices to $U'$ and $4\beta k_2s$ vertices to $V'$, we obtain two subsets of size $k_1s$ and $k_2s$ respectively, with at most $dk_1sk_2s+4\beta k_1sk_2s+4\beta k_1sk_2s<\alpha k_1sk_2s$ edges, and thus \eqref{extremalcondition} holds in $G$.
\end{proof}

For the rest of this proof, we suppose that \eqref{extremalcondition} does not hold in $G$.

\begin{claim}\label{matching}
$G_r$ contains a perfect matching.
\end{claim}

\begin{proof}
Let $M$ be a maximum matching of $G_r$. After relabeling indices if necessary, we may assume that
$M = \{U_iV_i : i\in[k], k\leq t\}$. If $M$ is not perfect, let $x\in \mathcal{U}$ and $y \in \mathcal{V}$ be vertices which are unsaturated by $M$. Then the neighborhood $N(x)$ is a subset of $V(M)$, otherwise we can enlarge $M$ by
adding an edge $xz$ for any $z \in N(x)\setminus V(M)$. We have $N(y) \subseteq V (M)$ for the same reason.
Now let $I = \{i : V_i \in N(x)\}$ and $J = \{j : U_j\in N(y)\}$. If $I \cap J\neq \emptyset$; that is, there exists $i$
such $xV_i$ and $yU_i$ are both edges, then we can obtain a larger matching by replacing $U_iV_i$ in
$M$ by $xV_i$ and $yU_i$. Otherwise, assume that $I \cap J = \emptyset$. Since $|I|\geq (\gamma_1-2\beta)t$ and $|J|\geq (\gamma_2-2\beta)t$
and \eqref{extremalcondition} does not hold in $G$, then by the contrapositive of Claim \ref{split} there exists an edge
between $\{U_i : i \in I\}$ and $\{V_j : j \in J\}$. This implies that there exist $i \neq j$ such that $xV_i$,
$U_iV_j$, and $yU_j$ are edges. Replacing $U_iV_i$, $U_jV_j$ in $M$ by $xV_i$, $U_iV_j$ and $yU_j$, we obtain a larger
matching, contradicting the maximality of $M$.

\end{proof}

By Claim \ref{matching} we assume that $U_i\sim V_i$ for all $i\in [t]$.  If each $\ep$-regular pair $(U_i, V_i)$ is also super-regular and $s$ divides $\ell$, then the Blow-up Lemma (Lemma \ref{blowup}) guarantees that $G'[U_i, V_i]$ can be tiled with $K_{s,s}$ (since $K_{\ell,\ell}$ can be tiled with $K_{s,s}$). If we also know that $U_0=V_0=\emptyset$, then we obtain a $K_{s,s}$-tiling of $G$. Otherwise we do the following steps (details of these steps are given next). \emph{Step 1}: For each $i\geq 1$, we move vertices from $U_i$ to $U_0$ and from $V_i$ to $V_0$ so that each remaining vertex in $(U_i, V_i)$ has at least $(d-2\ep)\ell$ neighbors. \emph{Step 2:} We eliminate $U_0$ and $V_0$ by removing copies of $K_{s,s}$, each of which contains at most one vertex of $U_0\cup V_0$. \emph{Step 3}: We make sure that for each $i\geq 1$, $|U_i|=|V_i| > (1 - d)\ell$ and $|U_i|$ is divisible by
$s$. Finally we apply the Blow-up Lemma to each $(U_i, V_i)$ (which is still super-regular) to finish the proof. Note that we always refer to the clusters as $U_i, V_i, i \geq 0$ even though they may gain or lose vertices during the process.

\emph{Step 1.} For each $i \geq 1$, we remove all $u\in U_i$ such that $\deg(u, V_i) < (d-\ep)\ell$ and all
$v\in V_i$ such that $\deg(v, U_i) < (d-\ep)\ell$. Fact \ref{interprop} (with $k=1$) guarantees that the
number of removed vertices is at most $\ep\ell$. We then remove more vertices from either $U_i$
or $V_i$ to make sure $U_i$ and $V_i$ still have the same number of vertices. All removed vertices
are added to $U_0$ and $V_0$. As a result, we have $|U_0|=|V_0|\leq 2\ep n$.

\emph{Step 2.} This step implies that a vertex in $U_0, V_0$ can be viewed as a vertex in $U_i$ or $V_i$ for
some $i\geq 1$. For a vertex $x \in V(G)$ and a cluster $C$, we say $x$ is adjacent to $C$, denoted $x \sim C$, if
$\deg_G(x,C)\geq d\ell$. We claim that at present, each vertex in $U$ is adjacent to at least
$(\gamma_1-2\beta)t$ clusters.  If this is not true for some $u \in U$, then we obtain a contradiction
$$(\gamma_1-\beta)n\leq \deg_G(u)\leq (\gamma_1-2\beta)t\ell+d\ell t +2\ep n<(\gamma_1-3\beta/2)n.$$
Likewise, each vertex in $V$ is adjacent to at least $(\gamma_2-2\beta)t$ clusters.  Assign an arbitrary order to the vertices in $U_0$. For each $u \in U_0$, we pick some $V_i$ adjacent
to $u$. The selection of $V_i$ is arbitrary, but no $V_i$ is selected more than $\frac{d\ell}{6s}$ times. Such $V_i$
exists even for the last vertex of $U_0$ because $|U_0|\leq 2\ep n<(\gamma_1-2\beta)t\frac{d\ell}{6s}$. For each $u\in U_0$
and its corresponding $V_i$, we remove a copy of $K_{s,s}$ containing $u$, $s$ vertices in $V_i$, and $s-1$ vertices
in $U_i$. Such a copy of $K_{s,s}$ can always be found even if $u$ is the last vertex in $U_0$
because $(U_i, V_i)$ is $\ep$-regular and $\deg_G(u, V_i)\geq d\ell>\ep\ell+\frac{d\ell}{6s}s$ thus Fact \ref{interprop} (with $k=s-1$) allows us to choose $s-1$ vertices from $U_i$ and $s$ vertices from $N(u)\cap V_i$ to complete the copy of $K_{s,s}$. As a result, $U_i$ now has
one more vertex than $V_i$, so one may view this process as moving $u$ to $U_i$. We repeat this
process for all $v \in V_0$ as well. By the end of this step, we have $U_0=V_0=\emptyset$, and each
$U_i$, $V_i$, $i \geq 1$ contains at least $\ell -\ep\ell-d\ell/3$ vertices (for example, $U_i$ may have lost $\frac{d\ell(s-1)}{6s}$
vertices because of $U_0$ and $d\ell/6$ vertices because of $V_0$).  As a result, we have $\delta(G[U_i, V_i])\geq (\frac{2d}{3}-2\ep)\ell$ for all $i\geq 1$. Note that the sizes of $U_i$ and $V_i$ may currently be different.

\emph{Step 3}. We want to show that for any $i \neq j$, there is a path $U_iV_{i_1}U_{i_1}\dots V_{i_a}U_{i_a}V_jU_j$ (resp. $V_iU_{i_1}V_{i_1}\dots U_{i_a}V_{i_a}U_jV_j$) for some $0\leq a\leq 2$. If such a path exists, then for each $i_b$, $1\leq b\leq a+1$ (assume that $i = i_0$ and $j = i_{a+1}$), we may remove a copy of $K_{s,s}$ containing one vertex from $U_{i_{b-1}}$, $s$ vertices
from $V_{i_b}$, and $s-1$ vertices from $U_{i_b}$. This removal reduces the size of $U_i$ by one, increases
the size of $U_j$ by one but does not change the sizes of other clusters (all modulo $s$). We may
therefore adjust the sizes of $U_i$ and $V_i$ (for $i \geq 1$) such that $|U_i|=|V_i|$ and $|U_i|$ is divisible
by $s$.  To do this we will need at most $2t$ paths: (i) Let $r:=\floor{\frac{n}{t}}\hspace{-.1in}\mod s$. (ii) Pair up the current biggest set $U_i$ and current smallest set $U_j$ and move vertices from $U_i$ to $U_j$ until one of the sets has exactly $\floor{\frac{n}{t}}-r$ elements. (iii) Repeat this process until all but one set in $\mathcal{U}$ has exactly $\floor{\frac{n}{t}}-r$ elements (there will be one set, say $U_t$, with as many as $(t-1)^2$ extra vertices) (iv) Do the same for the clusters in $\mathcal{V}$.  

Now we show how to find this path from $U_1$ to $U_2$.  First, if $U_1\sim V_2$, then $U_1V_2U_2$
is a path. Let $I = \{i : U_1 \sim V_i\}$ and $J = \{i : U_i\sim V_2\}$. If there exists $i \in I \cap J$, then
we find a path $U_1V_iU_iV_2U_2$. Otherwise $I \cap J = \emptyset$. Since both $|I|\geq (\gamma_1-2\beta)t$ and $|J|\geq (\gamma_2-2\beta)t$, Claim \ref{split} guarantees that there exists $i\in I$ and $j\in J$ such that $U_i\sim V_j$. We thus have a path $U_1V_iU_iV_jU_jV_2U_2$. Note that in this step we require that a cluster is contained in at most $\frac{d\ell}{3s}$ paths. This restriction has little impact
on the arguments above: we have $|I|>(\gamma_1-3\beta)t$ and $|J| > (\gamma_2-3\beta)t$ instead, still satisfying the conditions
of Claim \ref{split}.  

Now $U_0=V_0=\emptyset$, and for all $i \geq 1$, $|U_i|=|V_i|$ is divisible by $s$.  Let $\mathcal{K}$ be the union of all vertices in existing copies of $K_{s,s}$ and note that, 
\begin{equation*}
|U_i\setminus \mathcal{K}|=|V_i\setminus \mathcal{K}|\geq \ell-\ep \ell-2d\ell/3,
\end{equation*} 
which implies $\delta(G[U_i, V_i])\geq (\frac{d}{3}-2\ep)\ell\geq \frac{d}{4}\ell$ for $i\geq 1$.  Thus Fact \ref{slicing} implies that each pair $(U_i, V_i)$ is $(2\ep, \frac{d}{4})$-super-regular.  Applying the Blow-up Lemma to each $(U_i, V_i)$, we find the desired $K_{s,s}$-tiling.

\end{proof}

\section{Extremal Case}\label{extremesection}

In this section we prove Theorems \ref{main 1} and \ref{main 2} in the case when $G$ satisfies the extremal condition.  

Given $s\geq 2$ and $\lambda\in (0,\frac{1}{2})$, let $\alpha>0$ be sufficiently small.  Let $G[U,V]$ be a balanced bipartite graph on $2n=2ms$ vertices for sufficiently large $n$.  Without loss of generality suppose $\delta_V\geq \delta_U$ and note that $\delta_U\geq \lambda n$.  Suppose $G$ is edge minimal with respect to the condition $\delta_U+\delta_V\geq n+c$, and that $G$ satisfies the extremal condition with parameter $\alpha$.  Let $k_1$ be defined by $\delta_U=k_1s+s+r$, where $0\leq r\leq s-1$ and let $k_2s=n-k_1s$.

The proof will split into cases depending on whether $k_1\leq (1-\frac{1}{2s})k_2$ (we say $k_2\gg k_1$) or $k_1> (1-\frac{1}{2s})k_2$ (we say $k_1\approx k_2$).  When $k_1> (1-\frac{1}{2s})k_2$, we are only dealing with Theorem \ref{main 1} in which case we have $\delta_U+\delta_V\geq n+3s-5$.  Since $\delta_U=k_1s+s+r$, we have $\delta_V\geq k_2s+2s-5-r$. Since $G$ is edge minimal we have $\delta_V=k_2s+2s-5-r$, and since $\delta_V\geq \delta_U$, we have $k_2\geq k_1$.  If $\delta_V=\delta_U$, then we have $$\delta(G)\geq \frac{n+3s-5}{2}> \left\lbrace \begin{array}{ll} \frac{n}{2}+s-2   & \text{ if } m \text{ is even } \\ \frac {n+3s}{2}-3 & \text{ if } m \text{ is odd, } \end{array} \right. $$ which is solved in \cite{Z}.  So we may suppose that $\delta_V>\delta_U$.  

\begin{claim}\label{k2approxk1}
If $k_2=k_1$, then $r\leq \frac{s-6}{2}$ and consequently $\delta_V=k_2s+2s-5-r\geq k_2s+s$.  If $k_2=k_1+1$, then $r\leq s-3$ and consequently $\delta_V=k_2s+2s-5-r\geq k_2s+s-2$.
\end{claim}

\begin{proof}
Both statements are implied the following inequality: $k_2s+2s-5-r=\delta_V>\delta_U=k_1s+s+r$.
\end{proof}

When $k_1\leq (1-\frac{1}{2s})k_2$, we either have $k_2<(s-d)k_1$, in which case we are still only dealing with Theorem \ref{main 1} and we will assume $\delta_U+\delta_V\geq n+3s-5$, or we have $k_2\geq (s-d)k_1$, in which case we are dealing with Theorem \ref{main 2} and we will assume $\delta_U+\delta_V\geq \delta_U+\delta_V\geq n+2s-2\croot{s}+d+c(s)$.  

\subsection{Pre-processing}\label{preprocesssection}

Let $U_2'=U\setminus U_1'$ and $V_1'=V\setminus V_2'$.  Let
\begin{align*}
U_1&=\{x\in U: 
\deg(x,V_2')<\alpha^{1/3}k_1s\},~~~
V_2=\{x\in V: 
\deg(x,U_1')<\alpha^{1/3}k_2s\},\\
U_2&=\{x\in U: \deg(x, V_1')<\alpha^{1/3}k_1s \vee \deg(x,V_2')>(1-\alpha^{1/3})k_2s\},\\
V_1&=\{x\in V: \deg(x,U_2')<\alpha^{1/3}k_2s \vee \deg(x, U_1')>(1-\alpha^{1/3})k_1s\},\\
U_0&=U\setminus (U_1\cup U_2), \text{ and } V_0=V\setminus (V_1\cup V_2).\\
\end{align*}

\begin{claim}\label{preprocess}
\begin{enumerate}
\item $k_1s-\alpha^{2/3}k_2s\leq|U_1|,|V_1|\leq k_1s+\alpha^{2/3}k_1s$

\item $k_2s-\alpha^{2/3}k_1s\leq|U_2|,|V_2|\leq k_2s+\alpha^{2/3}k_2s$

\item $|U_0|,|V_0|\leq \alpha^{2/3}n$

\item $\delta(U_0,V_1)\geq \alpha^{1/3}k_1s-\alpha^{2/3}k_2s$, $\delta(U_0,V_2)\geq \alpha^{1/3}k_1s-\alpha^{2/3}k_1s$

\item $\delta(V_0,U_1)\geq \alpha^{1/3}k_2s-\alpha^{2/3}k_2s$, $\delta(V_0,U_2)\geq \alpha^{1/3}k_2s-\alpha^{2/3}k_1s$

\item $\delta(G[U_i,V_i])\geq k_is-\alpha^{1/3}k_is-\alpha^{2/3}k_{3-i}s\geq (1-2\alpha^{1/3})k_is$

\item $\Delta(U_1,V_2)\leq 2\alpha^{1/3}k_1s$, $\Delta(V_2,U_1)\leq 2\alpha^{1/3}k_2s$
\end{enumerate}

\end{claim}

\begin{proof}
We have 
\begin{align*}
\alpha^{1/3}k_1s|U_1'\setminus U_1|\leq e(U_1'\setminus U_1,V_2')\leq e(U_1',V_2')\leq \alpha k_1sk_2s 
\end{align*}
which gives $|U_1'\setminus U_1|\leq \alpha^{2/3}k_2s$ and thus $|U_1|\geq k_1s-\alpha^{2/3}k_2s$.

Also 
\begin{align*}
\alpha^{1/3}k_2s|V_2'\setminus V_2|\leq e(V_2'\setminus V_2,U_1')\leq e(V_2',U_1')\leq \alpha k_1sk_2s
\end{align*}
which gives $|V_2'\setminus V_2|\leq \alpha^{2/3}k_1s$ and thus $|V_2|\geq k_2s-\alpha^{2/3}k_1s$.

Since $e(U_1',V_2')\leq \alpha k_1sk_2s$, we have $e(U_2',V_2')\geq k_2sk_2s-\alpha k_1sk_2s$ and $e(U_1',V_1')\geq k_1sk_1s-\alpha k_1sk_2s$.  Thus 
\begin{align*}
\alpha^{1/3}k_2s|U_2'\setminus U_2|\leq \bar{e}(U_2',V_2')\leq \alpha k_1sk_2s
\end{align*}
which gives $|U_2'\setminus U_2|\leq \alpha^{2/3}k_1s$ and thus $|U_2|\geq k_2s-\alpha^{2/3}k_1s$.

Also
\begin{align*}
\alpha^{1/3}k_1s|V_1'\setminus V_1|\leq \bar{e}(U_1',V_1')\leq \alpha k_1sk_2s
\end{align*}
which gives $|V_1'\setminus V_1|\leq \alpha^{2/3}k_2s$ and thus $|V_1|\geq k_1s-\alpha^{2/3}k_2s$.

Putting these results together we have $|U_0|,|V_0|\leq \alpha^{2/3} n$, $|U_1|,|V_1|\leq k_1s+\alpha^{2/3}k_1s$, and $|U_2|,|V_2|\leq k_2s+\alpha^{2/3}k_2s$.

By the definition of $U_1,U_2,V_1,V_2$ and the lower bounds on their sizes, we have $\delta(U_0,V_1)\geq \alpha^{1/3}k_1s-\alpha^{2/3}k_2s$, $\delta(U_0,V_2)\geq \alpha^{1/3}k_1s-\alpha^{2/3}k_1s$, $\delta(V_0,U_1)\geq \alpha^{1/3}k_2s-\alpha^{2/3}k_2s$, and $\delta(V_0,U_2)\geq \alpha^{1/3}k_2s-\alpha^{2/3}k_1s$.  By the definition of $U_1,V_2$ and the upper bounds on their sizes we have $\Delta(U_1,V_2)\leq 2\alpha^{1/3}k_1s$ and $\Delta(V_2,U_1)\leq 2\alpha^{1/3}k_2s$.

\end{proof}

\subsection{Idea of the Proof}

We start with the partition given in Section \ref{preprocesssection} and we call $U_0$ and $V_0$ the \emph{exceptional} sets. Let $i\in \{1,2\}$. We will attempt to update the partition by moving a constant number (depending only on $s$) of \emph{special} vertices between $U_1$ and $U_2$, denote them by $X$, and \emph{special} vertices between $V_1$ and $V_2$, denote them by $Y$, as well as partitioning the exceptional sets as $U_0=U_0^1\cup U_0^2$ and $V_0=V_0^1\cup V_0^2$. Let $U_1^*$, $U_2^*$, $V_1^*$ and $V_2^*$ be the resulting sets after moving the special vertices. 
Suppose $u$ is a special vertex in the set $U_1^*$.  The degree of $u$ in $V_1^*$ may be small, but $u$ will have a set of at least $s$ neighbors in $V_1^*$ which are disjoint from the neighbors of any other special vertex in $U_1^*$.  Furthermore, these neighbors of $u$ in $V_1^*$ will have huge degree in $U_1^*$, so it will be easy to incorporate each special vertex into a unique copy of $K_{s,s}$.

Our goal is to obtain two graphs, $G_1:=G[U_1^*\cup U_0^1, V_1^*\cup V_0^1]$ and $G_2:=[U_2^*\cup U_0^2, V_2^*\cup V_0^2]$ so that $G_1$ satisfies $$|U_1^*\cup U_0^1|=\ell_1s,~ |V_1^*\cup V_0^1|=\ell_1s$$ and $G_2$ satisfies $$|U_2^*\cup U_0^2|=\ell_2s,~ |V_2^*\cup V_0^2|=\ell_2s,$$ for some positive integers $\ell_1,\ell_2$.  We tile $G_1$ as follows.  We incorporate all of the special vertices into copies of $K_{s,s}$.  We now deal with the exceptional vertices: Claim \ref{preprocess} gives $|U_0|, |V_0|\leq \alpha^{2/3}n$ and $\delta(U_0, V_i), \delta(V_0, U_i)\gg s\alpha^{2/3}n$, so they may greedily be incorporated into unique copies of $K_{s,s}$. Then we are left with two balanced ``almost complete'' graphs, which can be easily tiled.

So throughout the proof, if we can make, say $|U_1^*\cup U_0^1|$ and $|V_1^*\cup V_0^1|$ equal and divisible by $s$, we simply state that ``we are done.''

\subsection{Preliminary Lemmas}

In this section we give some lemmas which will be used in the proof of Theorems \ref{main 1} and \ref{main 2}.  Recall that in each of those theorems we suppose $k_2s\geq k_1s\geq \lambda n$.  

\begin{lemma}[Zhao \cite{Z}, Fact 5.3]\label{Zhao lemma}
Let $F$ be an $A,B$-bigraph with $\delta:=\delta(A,B)$ and $\Delta:=\Delta(B,A)$ 
Then $F$ contains $f_h$ vertex disjoint $h$-stars from $A$ to $B$, and $g_h$ vertex disjoint $h$-stars from $B$ to $A$ (the stars from $A$ to $B$ and those from $B$ to $A$ need not be disjoint), where
\begin{align*}
f_h\geq\frac{(\delta-h+1)|A|}{h\Delta+\delta-h+1}, ~~~ g_h\geq\frac{\delta|A|-(h-1)|B|}{\Delta+h\delta-h+1}.
\end{align*}

\end{lemma}

\begin{lemma}\label{STARSlemma}
Let $G[A,B]$ be a bipartite graph with 
$|B|=\ell s+b$ for some positive integers $\ell$ and $b$.  Let $0\leq x\leq s-1$ and let $\gamma$ be a small constant such that $\alpha^{1/3}\ll \gamma\ll \frac{1}{2s}$.  If $b<\frac{1}{\gamma}$ and
\begin{enumerate}
\item $\delta(B,A)\geq s-x$, $\Delta(A,B)\leq 2\alpha^{1/3}k_2s$, and $|B|\geq \alpha^{1/6}|A|$
\end{enumerate}
then there are at least $b$ vertex disjoint $(s-x)$-stars from $B$ to $A$.

Suppose $k_2s+\alpha^{2/3}k_2s\geq |A|, |B|\geq k_1s-\alpha^{2/3}k_2s$.  If
\begin{enumerate}[resume]
\item $\delta(A,B)\geq s-1+b$ and $k_1>(1-\frac{1}{2s})k_2$,
\end{enumerate}
then there are at least $b$ vertex disjoint $s$-stars from $B$ to $A$.  If $b<\frac{1}{\gamma}$ and
\begin{enumerate}[resume]

\item $\delta(A,B)\geq s$, $k_1>(1-\frac{1}{2s})k_2$, and $\Delta(B,A)\leq 2\alpha^{1/3}k_2s$ or

\item $\delta(A,B)\geq d$, $|A|\geq \frac{s-1/2}{d}|B|$, and $\Delta(B,A)\leq 2\alpha^{1/3}k_2s$,
\end{enumerate}
then there are at least $b$ vertex disjoint $s$-stars from $B$ to $A$.  Furthermore, if $b\geq \frac{1}{\gamma}$ and 
\begin{enumerate}[resume]
\item $\delta(A,B)\geq b/4$ and $\Delta(B,A)<2\alpha^{1/3}k_2s$ or

\item $\delta(B,A)\geq b/4$ and $\Delta(A,B)<2\alpha^{1/3}k_2s$,

\end{enumerate}
then there are at least $b$ vertex disjoint $s$-stars from $B$ to $A$.

\end{lemma}

\begin{proof}

\begin{enumerate}

\item Suppose $b<\frac{1}{\gamma}$, $\delta(B,A)\geq s-x$, $\Delta(A,B)\leq 2\alpha^{1/3}k_2s$, and $|B|\geq \alpha^{1/6}|A|$.  Let $\mathcal{S}_B$ be the maximum set of vertex disjoint $(s-x)$-stars from $B$ to $A$ and let $f_{s-x}=|\mathcal{S}_B|$.  By Lemma \ref{Zhao lemma}, we have
\begin{align*}
f_{s-x}\geq \frac{|B|}{2(s-x)\alpha^{1/3}k_2s+1}\geq \frac{\alpha^{1/6}}{3s\alpha^{1/3}}
\geq \frac{1}{\gamma}
\geq b
\end{align*}

\item Suppose $\delta(A,B)\geq s-1+b$ and $k_1>(1-\frac{1}{2s})k_2$.  Let $\mathcal{S}_A$ be a maximum set of vertex disjoint $s$-stars with centers $C\subseteq B$ and leaves $L\subseteq A$.  Suppose $|C|\leq b-1$.  Then 
\begin{align*}
s(|A|-|L|) \leq(s-1+b-|C|)(|A|-|L|) \leq e(B\setminus C, A\setminus L)\leq (s-1)(|B|-|C|),
\end{align*}
which implies
\begin{align*}
s(k_1s-\alpha^{2/3}k_2s)\leq (s-1)(k_2s+\alpha^{2/3}k_2s)+s|L|-(s-1)|C|.
\end{align*}
Thus $sk_1\leq (s-\frac{1}{2})k_2$, contradicting the fact that $k_1>(1-\frac{1}{2s})k_2$.

\item Suppose $b<\frac{1}{\gamma}$, $\delta(A,B)\geq s$, $k_1>(1-\frac{1}{2s})k_2$, and $\Delta(B,A)\leq 2\alpha^{1/3}k_2s$.  Let $\mathcal{S}_A$ be the maximum set of vertex disjoint $s$-stars from $A$ to $B$ and let $g_s=|\mathcal{S}_A|$.  By Lemma \ref{Zhao lemma}, we have
\begin{align*}
g_s\geq \frac{s|A|-(s-1)|B|}{2\alpha^{1/3}k_2s+s^2-s+1}\geq \frac{s(k_1s-\alpha^{2/3}k_2s)-(s-1)(k_2s+\alpha^{2/3}k_2s)}{3\alpha^{1/3}k_2s}
\geq \frac{1}{12\alpha^{1/3}}\geq \frac{1}{\gamma}\geq b
\end{align*}
Where the third inequality holds since $sk_1s> (s-\frac{1}{2})k_2s$.

\item Suppose $b<\frac{1}{\gamma}$, $\delta(A,B)\geq d$, $|A|\geq \frac{s-1/2}{d}|B|$, and $\Delta(B,A)\leq 2\alpha^{1/3}k_2s$.  Let $\mathcal{S}_B$ be the maximum set of vertex disjoint $s$-stars from $B$ to $A$ and let $g_s=|\mathcal{S}_B|$.  By Lemma \ref{Zhao lemma}, we have
\begin{align*}
g_s\geq \frac{d|A|-(s-1)|B|}{2\alpha^{1/3}k_2s+sd-s+1}\geq \frac{|B|/2}{3\alpha^{1/3}k_2s}
\geq \frac{\lambda}{6\alpha^{1/3}}
\geq \frac{1}{\gamma}
\geq b
\end{align*}

\item Suppose $b\geq \frac{1}{\gamma}$, $\delta(A,B)\geq b/4$ and $\Delta(B,A)<2\alpha^{1/3}k_2s$.  Let $\mathcal{S}_B$ be the maximum set of vertex disjoint $s$-stars from $B$ to $A$ and let $g_s=|\mathcal{S}_B|$.  By Lemma \ref{Zhao lemma}, we have
\begin{align*}
g_s\geq \frac{\frac{b}{4}|A|-(s-1)|B|}{2\alpha^{1/3}k_2s+s\frac{b}{4}-s+1}\geq \frac{b\lambda/4-(s-1)}{3\alpha^{1/3}}\geq b
\end{align*}

\item Suppose $b\geq \frac{1}{\gamma}$, $\delta(B,A)\geq b/4$ and $\Delta(A,B)<2\alpha^{1/3}k_2s$.  Let $\mathcal{S}_B$ be the maximum set of vertex disjoint $s$-stars from $B$ to $A$ and let $f_s=|\mathcal{S}_B|$.  By Lemma \ref{Zhao lemma}, we have
\begin{align*}
f_s\geq \frac{(\frac{b}{4}-s+1)|B|}{2s\alpha^{1/3}k_2s+\frac{b}{4}-s+1}\geq \frac{(\frac{b}{4}-s+1)\lambda}{3\alpha^{1/3}}\geq b
\end{align*}

\end{enumerate}

\end{proof}



\begin{lemma}\label{lemma:diagonalsum}
Let $G[A,B]$ be a bipartite graph with $|A|=\ell_1s+a$ and $|B|=\ell_2s+b$ such that 
$1\leq b\leq s-1$.  Suppose further that $k_2s+\alpha^{2/3}k_2s\geq |A|,|B|\geq k_1s-\alpha^{2/3}k_2s$ and $\Delta(A, B), \Delta(B, A)\leq 2\alpha^{1/3} k_2s$.  If 
\begin{enumerate}
\item $a\geq 1$ and $\delta(A, B)+\delta(B, A)\geq 2s-3+a+b$ or

\item $a=0$ and $\delta(A, B)+\delta(B, A)\geq 2s-2+b$,
\end{enumerate}
then there is a set $\mathcal{S}_A$ of $a$ vertex disjoint $s$-stars from $A$ to $B$ and a set $\mathcal{S}_B$ of $b$ vertex disjoint $s$-stars from $B$ to $A$ such that the stars in $\mathcal{S}_A$ are disjoint from the stars in $\mathcal{S}_B$.
\end{lemma}

\begin{proof}
Let $\gamma$ be a real number such that $\alpha^{1/3}\ll \gamma\ll \frac{1}{2s}$.  

\noindent
\textbf{Case 1} $a>\frac{1}{\gamma}$.  Suppose first $\delta(B, A)\geq \frac{1}{2}(2s-3+a+b)$.  In this case we apply Lemma \ref{STARSlemma}(vi) to get a set of $b$ vertex disjoint $s$-stars with centers $C\subseteq B$ and leaves $L\subseteq A$.  Then since $\delta(B, A\setminus L)\geq \frac{1}{2}(2s-3+a+b)-bs>\frac{a}{4}$ we apply Lemma \ref{STARSlemma}(v) to get a set of $a$ vertex disjoint $s$-stars from $A\setminus L$ to $B\setminus C$.  Now suppose $\delta(A, B)> \frac{1}{2}(2s-3+a+b)$.  As before, we apply Lemma \ref{STARSlemma}(v) to get a set of $b$ vertex disjoint $s$-stars with centers $C\subseteq B$ and leaves $L\subseteq A$.  Then since $\delta(A, B\setminus C)> \frac{1}{2}(2s-3+a+b)-b>\frac{a}{4}$ we apply Lemma \ref{STARSlemma}(vi) to get a set of $a$ vertex disjoint $s$-stars from $A\setminus L$ to $B\setminus C$.

\noindent
\textbf{Case 2} $1\leq a\leq \frac{1}{\gamma}$.  Suppose first that $\delta(B, A)\geq s-1+a$.  We apply Lemma \ref{STARSlemma}(ii) to get a set of $a$ vertex disjoint $s$-stars with centers $C\subseteq A$ and leaves $L\subseteq B$.  We still have $\delta(B\setminus N(C), A\setminus C)\geq s-1+a$ and $|B\setminus N(C)|\geq |B|-\frac{2\alpha^{1/3}}{\gamma}k_2s\geq \alpha^{1/6}|A|$, thus we can apply Lemma \ref{STARSlemma}(i) to get a set of $b$ vertex disjoint $s$-stars from $B\setminus N(C)$ to $A\setminus C$.  Now suppose $\delta(A, B)\geq s+b$.  We apply Lemma \ref{STARSlemma}(ii) to get a set of $b$ vertex disjoint $s$-stars with centers $C\subseteq B$ and leaves $L\subseteq A$.  We still have $\delta(A\setminus L, B\setminus C)\geq s+b-b=s$ so we apply Lemma \ref{STARSlemma}(i) to get $a$ vertex disjoint $s$-stars from $A\setminus L$ to $B\setminus C$.

\noindent
\textbf{Case 3} $a=0$.  We have $\delta(A, B)+\delta(B, A)\geq 2s-2+b\geq 2s-1$ and thus $\delta(A, B)\geq s$ or $\delta(B, A)\geq s$.  In either case we can apply Lemma \ref{STARSlemma}(i) or (iii) to get a set of $b$ vertex disjoint $s$-stars from $B$ to $A$.

\end{proof}

In addition, we will use the following fact from \cite{CD}.

\begin{lemma}\label{K_1sum}
Suppose $|U_0|\geq s$.  Let $V_1'\subseteq V_1$ and $V_2'\subseteq V_2$ such that $\delta(V_1', U_0)+\delta(V_2', U_0)\geq |U_0|+s$.  If $|V_1'|\geq\frac{n}{8}$ and $|V_2'|\geq\frac{n}{8}$, then for any $0\leq b\leq s$, there is a $K_{s,s}=:K$ with $s$ vertices in $U_0$, $b$ vertices in $V_1$ and $s-b$ vertices in $V_2$. 
\end{lemma}

\subsection{Case $k_2\gg k_1$}

In this section we prove Theorem \ref{main 2} and prove Theorem \ref{main 1} in the case that $k_1\leq (1-\frac{1}{2s})k_2$.  Let $G$ be a graph which satisfies the extremal condition and for which $k_1\leq (1-\frac{1}{2s})k_2$.  Recall the bounds from Claim \ref{preprocess}, specifically $k_1s-\alpha^{2/3}k_2s\leq|U_1|,|V_1|\leq k_1s+\alpha^{2/3}k_1s$, $k_2s-\alpha^{2/3}k_1s\leq|U_2|,|V_2|\leq k_2s+\alpha^{2/3}k_2s$, and $|U_0|,|V_0|\leq \alpha^{2/3}n$.  The fact that $\delta_U+\delta_V\geq n$ implies 
\begin{equation}\label{V1toU2:big}
\delta(V_1,U_2)\geq \delta_V-|U_0\cup U_1|\geq (k_2-k_1-2\alpha^{2/3}k_1)s
\geq(\frac{1}{2s} k_2-2\alpha^{2/3}k_1)s>\frac{1}{4s}k_2s.
\end{equation}


%

\begin{proof}
Note that $s-2\croot{s}+c(s)+1\geq 0$ with equality if and only if $s=2$, so $d$ is defined for all $s\geq 2$.  Let $\alpha^{1/3}\ll \gamma\ll \frac{1}{2s}$.  Let $\ell_1$ be maximal so that $|U_1|\geq \ell_1 s$ and $|V_0\cup V_1|\geq \ell_1 s$.  Let $y:=|U_1|-\ell_1 s$ and $z:=|V_0\cup V_1|-\ell_1 s$.  We note that $n+3s-5\geq n+2s-2\croot{s}+d+c(s)$ with equality if and only if $s=2$.  So for this proof we will assume $\delta_U+\delta_V\geq n+2s-2\croot{s}+d+c(s)$ with one exception that we point out.

\begin{claim}\label{V1>U1}
If there exists $\ell$ such that $|V_0\cup V_1|\geq \ell s$ and $|U_1|\leq \ell s$, then $G$ can be tiled with $K_{s,s}$.  
\end{claim}

\begin{proof}
Suppose there exists such an $\ell$.  By the choice of $\ell_1$, we can assume $|U_1|\leq (\ell_1+1)s$ and $|V_0\cup V_1|\geq (\ell_1+1)s$. By \eqref{V1toU2:big} we have $\delta(V_1,U_2)>\frac{1}{4s}k_2s\geq 2s\alpha^{2/3}n$ and thus we can greedily choose a set of $z-s$ vertex disjoint $s$-stars from $V_1$ to $U_2$ with centers $C_V$ and leaves $L_U$.  Let $V_1':=V_1\setminus C_V$ and $U_2':=U_2\setminus L_U$, since $\delta(V_1',U_2')\geq \frac{1}{8s} k_2s$ we may apply Lemma \ref{Zhao lemma} to the graph induced by $U_2'$ and $V_1'$ to get a set of $s-y$ vertex disjoint $s$-stars from $U_2'$ to $V_1'$.  We move the centers of the stars giving $|U_1|+(s-y)=(\ell_1+1)s=|V_0\cup V_1|-(z-s)$ and we are done.
\end{proof}

If $z\geq s$, then by the maximality of $\ell_1$ we have $y<s$ and thus we can apply Claim \ref{V1>U1} to finish.  If $y=0$, then we can also apply Claim \ref{V1>U1} to finish.  So for the rest of the proof, suppose that $0\leq z\leq s-1$ and $1\leq y$.  Our goal is to show that there exists a set $\mathcal{S}_U$ of vertex disjoint $(s-x)$-stars from $U_1$ to $V_2$ such that $|V_0\cup V_1|-x|\mathcal{S}_U|\geq |U_1|-|\mathcal{S}_U|=\ell_1 s$ and a set $\mathcal{T}_V$ of vertex disjoint $s$-stars from $V_1$ to $U_2$ so that $|V_0\cup V_1|-x|\mathcal{S}_U|-|\mathcal{T}_V|=\ell_1 s$ for some $0\leq x\leq s-1$.
Since $\delta_U+\delta_V\geq n+2s-2\croot{s}+d+c(s)$, we have 
\begin{align}
\delta(U_1,V_2)+\delta(V_2,U_1)&\geq n+2s-2\croot{s}+d+c(s)-|V_0\cup V_1|-|U_0\cup U_2|\notag\\
&\geq 2s-2\croot{s}+d+c(s)+y-z\label{diagonal}
\end{align}

\noindent
\textbf{Case 1} $|U_1|-|V_0\cup V_1|>0$.

\textbf{Case 1.1} $y\geq \frac{1}{\gamma}$. We have 
\begin{align*}
\delta(U_1,V_2)+\delta(V_2,U_1)\geq2s-2\croot{s}+d+c(s)+y-z\geq y+s-2\croot{s}+d+c(s)+1
\end{align*}
and thus there are two cases. Either $\delta(U_1,V_2)\geq \frac{1}{2}(y+s-2\croot{s}+d+c(s)+1)$ and we apply Lemma \ref{STARSlemma}(vi) to get $y$ vertex disjoint $s$-stars from $U_1$ to $V_2$ or $\delta(V_2,U_1)> \frac{1}{2}(y+s-2\croot{s}+d+c(s)+1)$ and we apply Lemma \ref{STARSlemma}(v) to get $y$ vertex disjoint $s$-stars from $U_1$ to $V_2$.  We move the centers from $U_1$ to $U_2$ to make $|U_1|=\ell_1 s$.  Then we move vertices from $V_0\cup V_1$ to $V_2$ to make $|V_0\cup V_1|=\ell_1 s$.

\textbf{Case 1.2} $y<\frac{1}{\gamma}$.


\textbf{Case 1.2.1.} $\delta(U_1,V_2)\geq s$. Apply Lemma \ref{STARSlemma}(i) with $x=0$ to get $y$ vertex disjoint $s$-stars from $U_1$ to $V_2$.

\textbf{Case 1.2.2.} $\delta(U_1,V_2)\leq s-1$.  By \eqref{diagonal} we have $\delta(V_2,U_1)\geq 2s-2\croot{s}+d+c(s)+y-z-(s-1)=s-2\croot{s}+d+c(s)+1+y-z\geq d+1$.  Since $k_2\geq (s-d)k_1$ and thus $|V_2|\geq (s-\frac{1}{2}-d)|U_1|\geq \frac{s-\frac{1}{2}}{d+1}|U_1|$, we can apply Lemma \ref{STARSlemma}(iv) to get $y$ vertex disjoint $s$-stars from $U_1$ to $V_2$.

\textbf{Case 2.} $|U_1|-|V_0\cup V_1|\leq 0$.  In this case we have $y\leq z$.  Rearranging \eqref{diagonal} gives 
\begin{equation}
\delta(U_1,V_2)+\delta(V_2,U_1)\geq 2s-2\croot{s}+d+c(s)-(z-y). \label{z-y}
\end{equation} 
Also since $k_1\leq \frac{k_2}{s-d}$, we have
\begin{align}
\delta(V_1, U_2)\geq \delta_V-|U_0\cup U_1|\geq (k_2-k_1-2\alpha^{2/3}k_1)s\geq(1-\frac{1+2\alpha^{2/3}}{s-d})k_2s 
&\geq  \frac{s-d-1-2\alpha^{2/3}}{(s-d)(1+\alpha^{2/3})}|U_2| \notag \\
&\geq \frac{s-d-1-\alpha^{1/3}}{s-d}|U_2|\label{V_1U_2}
\end{align}  

If $\delta_U+\delta_V\geq n+3s-5$, then \eqref{z-y} gives $\delta(U_1,V_2)+\delta(V_2,U_1)\geq 2s-3$ since $z-y\leq s-2$. Thus we have  $\delta(V_2, U_1)\geq s-1$ or $\delta(U_1, V_2)\geq s-1$.  In either case we can get $y$ vertex disjoint $(s-1)$-stars from $U_1$ to $V_2$ by Lemma \ref{STARSlemma}(iii) or Lemma \ref{STARSlemma}(i) with $x=1$.  For each $(s-1)$-star we choose a vertex from $V_1$ and $(s-1)$-vertices in $U_2$, which is possible by \eqref{V_1U_2} and $z\geq y$.  So for the rest of the proof we assume $\delta_U+\delta_V\geq n+2s-2\croot{s}+d+c(s)$.


\textbf{Case 2.1.} $z-y\leq s-2\croot{s}+c(s)+1$.  

\textbf{Case 2.1.1.} $\delta(U_1,V_2)\geq s-1$. We can get $y$ vertex disjoint $(s-1)$-stars from $U_1$ to $V_2$ by Lemma \ref{STARSlemma}(i) with $x=1$.  For each $(s-1)$-star we choose a vertex from $V_1$ and $(s-1)$-vertices in $U_2$, which is possible by \eqref{V_1U_2} and $z\geq y$.

\textbf{Case 2.1.2.} $\delta(U_1,V_2)\leq s-2$. So \eqref{z-y} and the condition of Case 2.2.1. gives $$\delta(V_2,U_1)\geq 2s-2\croot{s}+d+c(s)-(s-2\croot{s}+c(s)+1)-(s-2)=d+1.$$  We can get $y$ vertex disjoint $s$-stars from $U_1$ to $V_2$ by Lemma \ref{STARSlemma}(iv) as in Case 1.2.2.

\textbf{Case 2.2.} $z-y\geq s-2\croot{s}+c(s)+2$.  If $\delta(U_1,V_2)\geq s-1$ or $\delta(V_2,U_1)\geq d+1$, then we would be done as in the previous two cases.  So suppose $\delta(U_1,V_2)\leq s-2$ and $\delta(V_2,U_1)\leq d$ .  By \eqref{z-y}, we have
\begin{align}
s-2\geq s-x=\delta(U_1,V_2)&\geq 2s-2\croot{s}+d+c(s)-(z-y)-\delta(V_2, U_1)\label{s-x} \\
&\geq s-2\croot{s}+c(s)+2\geq d+1 \notag
\end{align}
for some $2\leq x\leq s-d-1$.

Let $\mathcal{S}_U$ be a set of $y$ vertex disjoint $(s-x)$-stars from $U_1$ to $V_2$, which exists by Lemma \ref{STARSlemma}(i).  For each $(s-x)$-star in $\mathcal{S}_U$ we will choose $s-1$ vertices from $U_2$ and $x$ vertices from $V_1$ to complete a copy of $K_{s,s}$. Let $u_1$ be the center of a star in $\mathcal{S}_U$ and let $v_1^1,v_1^2,\dots, v_1^x$ be a set of $x$ vertices in $N(u_1)\cap V_1$. By \eqref{V_1U_2}, we have $|N(v_1^1, v_1^2, \dots, v_1^x)\cap U_2|\geq \left(1-\frac{x(1+\alpha^{1/3})}{s-d}\right)|U_2|$.  Let $v_2^1,v_2^2,\dots, v_2^{s-x}$ be a set of $s-x$ vertices in $V_2$. By Claim \ref{preprocess}, we have $|N(v_2^1, v_2^2,\dots,v_2^{s-x})\cap U_2|\geq (1-(s-x)\alpha^{1/3})|U_2|$.  Thus 
$$ |N(v_1^1, v_1^2, \dots, v_1^x,v_2^1,v_2^2,\dots, v_2^{s-x})\cap U_2|\geq \left(1-\frac{x(1+\alpha^{1/3})}{s-d}-(s-x)\alpha^{1/3}\right)|U_2| \geq \alpha|U_2|$$ and we can choose $x$ vertices from $V_1$ and $s-1$ vertices from $U_2$ to turn each $s-x$ star into a copy of $K_{s,s}$.

Finally we must be sure that $|V_0\cup V_1|-xy\geq \ell s$, i.e. $z\geq xy$.  There are two cases.

\textbf{Case 2.2.1.} $1\leq q\leq p$ and consequently $c(s)=1$.  By \eqref{s-x} and $\delta(V_2, U_1)\leq d$, we get 
\begin{equation}
x+y\leq z-(s-2\croot{s}+1)\label{x+y first}
\end{equation}  
and thus
\begin{align*}
xy\leq\left(\frac{z-(s-2\croot{s}+1)}{2}\right)^2\leq z.
\end{align*}
The first inequality is by \eqref{x+y first} and the arithmetic mean-geometric mean inequality.  To verify the second inequality, let $F(z)=z-\left(\frac{z-(s-2\croot{s}+1)}{2}\right)^2$ and note $s-2\croot{s}+3\leq z\leq s-1$.  Using calculus, we see that $F$ achieves a maximum at $s-2\croot{s}+3$, $F$ is decreasing on the interval $[s-2\croot{s}+3, s-1]$ and $F(s-1)=s-1-(\croot{s}-1)^2=p^2+q-1-p^2\geq 0$.

\textbf{Case 2.2.2.} $q=0$ or $p+1\leq q\leq 2p$ and consequently $c(s)=0$.  By \eqref{s-x} and $\delta(V_2, U_1)\leq d$, we get
\begin{equation}
x+y\leq z-(s-2\croot{s}).\label{x+y second}
\end{equation}  
If $z=s-1$, then \eqref{x+y second} gives $x+y\leq 2\croot{s}-1$.  Since $2\croot{s}-1$ is odd, we have
\begin{align*}
xy\leq\left(\frac{2\croot{s}}{2}\right)\left(\frac{2\croot{s}-2}{2}\right)=\croot{s}(\croot{s}-1)\leq s-1=z
\end{align*}
where the last inequality holds by the assumption of this case.  So we may assume $z\leq s-2$.  So we have
\begin{align*}
xy\leq\left(\frac{z-(s-2\croot{s})}{2}\right)^2\leq z.
\end{align*}
The first inequality holds by (\ref{x+y second}) and the arithmetic mean-geometric mean inequality.  To verify the second inequality, let $F(z)=z-\left(\frac{z-(s-2\croot{s})}{2}\right)^2$ and note $s-2\croot{s}+2\leq z\leq s-2$.  Using calculus, we see that $F$ achieves a maximum at $s-2\croot{s}+2$, $F$ is decreasing on the interval $[s-2\croot{s}+2, s-2]$ and $F(s-2)=s-2-(\croot{s}-1)^2$.  When $q=0$ we have $p\geq 2$, and thus $F(s-2)=s-2-(\croot{s}-1)^2=p^2-2-(p^2-2p+1)=2p-3\geq 1$.  When $q\geq p+1$, we have $F(s-2)=s-2-(\croot{s}-1)^2=p^2+q-2-p^2=q-2\geq 0$.

\end{proof}

\subsection{Case $k_2\approx k_1$}

We are left to prove Theorem \ref{main 1} when $k_1> (1-\frac{1}{2s})k_2$.  
The proof is split into two cases depending on whether $s=2$ or $s\geq 3$.  
The proof of the $s\geq 3$ case follows a similar structure as the $s=2$ case, however the case analysis is extremely long and detailed.

We start with a graph which satisfies the extremal condition after pre-processing.  For $i=1,2$, let $U_i^M=\{u\in U_i:\deg(u, V_{3-i})>\alpha^{1/3}n\}$ and $V_i^M=\{v\in V_i: \deg(v, U_{3-i})>\alpha^{1/3}n \}$.  We call these vertices \emph{movable}.  Note that $U_1^M=\emptyset=V_2^M$ by Claim \ref{preprocess}.

\subsubsection{Case $s=2$}
Let $\gamma$ be a real number such that $\alpha^{1/3}\ll \gamma\ll \frac{1}{2s}$.  We assume that $n=2m$ and $\delta_V>\delta_U$, thus $\delta_V \geq \frac{n}{2}+1$. As a result
\begin{equation}\label{twocommon}
\forall v, v'\in V, |N(v)\cap N(v')|\geq 2
\end{equation}
Furthermore, since $\delta_V\geq \frac{n}{2}+1$, and since there is some vertex $u\in U$ with $\deg(u, V)\leq \frac{n}{2}$, 
\begin{equation}\label{u*}
\exists u^* \in U \mbox{ such that } \deg(u^*, V)\geq \frac{n}{2}+2.  
\end{equation}

\noindent
\textbf{Case 1.} $U_0\cup U_2^M\neq \emptyset$ or $|U_2|$ is even.  There are two cases: (i) $|V_0\cup V_1|>|U_1|$ 
or (ii) $|V_2|\geq |U_0\cup U_2|$.  If (i) is the case 
there exists some $\ell_1\in \mathbb{N}$, $X\subseteq U_0\cup U_2^M$, and $Y\subseteq V_0\cup V_1^M$ such that $|U_1\cup X|=\ell_1s$, $|(V_0\cup V_1)\setminus Y|\geq \ell_1s$ and $|(V_0\cup V_1)\setminus Y|-|U_1\cup X|$ is as small as possible.  If $|(V_0\cup V_1)\setminus Y|-|U_1\cup X|=0$, then we are done.  Otherwise there are no movable vertices left in $(V_0\cup V_1)\setminus Y$.   If (ii) is the case, then there exists some $\ell_2\in \mathbb{N}$ and $X\subseteq U_0\cup U_2^M$ with $|X|\leq 1$ such that $|(U_0\cup U_2)\setminus X|=\ell_2s$, $|V_2|\geq \ell_2s$ and $|V_2|-|(U_0\cup U_2)\setminus X|$ is as small as possible.

Notice that in either case, we are either done or there are no movable vertices left in $(V_0\cup V_1)\setminus Y$ or $V_2$.  Because of this symmetry we can suppose without loss of generality that that (i) is the case.  We reset $U_1:=U_1\cup X$ , $U_0:=(U_0\cup U_2^M)\setminus X$, $U_2:=U_2\setminus U_2^M$, $V_1:=V_1\setminus Y$, and $V_0:=V_0\cup Y$.  Let $\ell_2=m-\ell_1$.  Let $a:=|V_1|-\ell_1s$.  If $a=0$, then we are done, so suppose $a\geq 1$.  Note that there are no movable vertices in $V_1$ or $U_2$.  We have \begin{equation}\label{bothdirections}\delta(V_1, U_0\cup U_2)+\delta(U_0\cup U_2, V_1)\geq a+1.\end{equation} 

\textbf{Case 1.1.} $a>\frac{1}{\gamma}$.  We know that $|U_0|\leq 1$, otherwise we could make $a$ smaller by moving $2$ vertices from $U_0$ to $U_1$ while maintaining the fact that $|U_1|$ is even.  Either $\delta(V_1, U_2)\geq \delta(V_1, U_0\cup U_2)-1\geq \frac{a+1}{2}-1$ and we apply Lemma \ref{STARSlemma}(vi) to get $a$ vertex disjoint $2$-stars from $V_1$ to $U_2$ or else $\delta(U_0\cup U_2, V_1)> \frac{a+1}{2}$ and we apply Lemma \ref{STARSlemma}(v) to get $a$ vertex disjoint $2$-stars from $V_1$ to $U_2$.  We move the centers from $V_1$ to $V_2$ to make $|V_1|=\ell_1 s$.

\textbf{Case 1.2.} $a\leq \frac{1}{\gamma}$. If $\delta(U_0\cup U_2, V_1)\geq 2$, then we apply Lemma \ref{STARSlemma}(iii) to get a set of $a$ vertex disjoint $2$-stars from $V_1$ to $U_2$. So suppose $\delta(U_0\cup U_2, V_1)\leq 1$ and thus \begin{equation}\label{V1toU2}\delta(V_1, U_0\cup U_2)\geq a.\end{equation}

\textbf{Case 1.2.1.} $a\geq 3$.  We know that $|U_0|\leq 1$, otherwise we could make $a$ smaller by moving $2$ vertices from $U_0$ to $U_1$ while maintaining the fact that $|U_1|$ is even.  Since $a\geq 3$, we have $\delta(V_1, U_2)\geq \delta(V_1, U_0\cup U_1)-1\geq 2$ by \eqref{V1toU2}, and thus we can apply Lemma \ref{STARSlemma}(i) to get a set of $a$ vertex disjoint $2$-stars from $V_1$ to $U_2$.
So we only need to deal with the case $a\leq 2$.  

\textbf{Case 1.2.2.} $a=2$.  If $U_0=\emptyset$, then we can use \eqref{V1toU2} and apply Lemma \ref{STARSlemma}(i) to get a set of $a$ vertex disjoint $2$-stars from $V_1$ to $U_2$.  So suppose $U_0=\{u_0\}$.  If there is a vertex $u\in U_2$ with $\deg(u, V_1)=0$, then by (\ref{bothdirections}) we have $\delta(V_1, U_0\cup U_2)\geq 3$ and we are done since $\delta(V_1, U_2)\geq \delta(V_1, U_0\cup U_1)-1\geq 2$. So suppose $\delta(U_0\cup U_2)\geq 1$.  If there is a vertex $u\in U_2$ with $\deg(u, V_1)\geq 2$, then we can move $u_0$ and $u$ to $U_1$, thus for all $u\in U_2$, $\deg(u, V_1)=1$.  Now suppose there is a vertex $v_1\in V_1$ with $\deg(v_1, U_2)\geq 2$ and let $u_2, u_2'\in N(v)\cap U_2$.  Let $v_1'\in N(u_0)\cap (V_1\setminus \{v_1\})$.  Since $\Delta(U_2, V_1)\leq 1$, there exists some $u'\in (U_2\setminus \{u_2, u_2'\})\cap N(v_1')$.  Thus we can move $v_1$ and $v_1'$.  So for all $v\in V_1$, $\deg(v, U_2)=1$.  This implies that $\ell_2s-1=|U_2|=|V_1|=\ell_1s+2$, a contradiction.

\textbf{Case 1.2.3.} $a=1$.  If $U_0\neq \emptyset$, then let $u_0\in U_0$.  Let $u_2v_1\in E(V_1, (U_0\cup U_2)\setminus \{u_0\})$, which exists be \eqref{bothdirections}.  Let $v_2\in N(u_2)\cap V_2$.  By \eqref{twocommon}, $v_1$ and $v_2$ have a common neighbor $u'$ different than $u_2$.  If $u'\in U_0\cup U_2$, then we are done by simply moving $v_1$, so we have $u'\in U_1$ which completes a $K_{2,2}$.  Now we move $u_0$ to $U_1$ to finish.

Finally, suppose $U_0=\emptyset$.  If there exists a vertex $v\in V_1$ such that $\deg(v, U_2)\geq 2$, then we can move $v$ and be done.  So suppose $\Delta(V_1, U_2)\leq 1$.  Furthermore if there was a vertex $v\in V_1$ such that $\deg(v, U_2)=0$, then \eqref{bothdirections} would imply $\delta(U_2, V_1)\geq 2$ contradicting the fact that $\Delta(V_1, U_2)\leq 1$.  So every vertex in $V_1$ has exactly one neighbor in $U_2$ and \eqref{bothdirections} implies $\delta(U_2, V_1)\geq 1$.  Since $|U_2|$ is even and $|V_1|$ is odd, we must have $|V_1|\neq |U_2|$.  If $|U_2|>|V_1|$, then $\delta(U_2, V_1)\geq 1$ would imply that there was a vertex in $V_1$ with two neighbors in $U_2$, so suppose $|V_1|>|U_2|$. This implies that there exists some $u_0\in U_2$ such that $\deg(u_0, V_1)\geq 2$. Let $u_2v_1\in E(V_1, U_2\setminus \{u_0\})$, which exists be \eqref{bothdirections}.  Let $v_2\in N(u_2)\cap V_2$.  By \eqref{twocommon}, $v_1$ and $v_2$ have a common neighbor $u'$ different than $u_2$.  If $u'\in U_2$, then we are done by simply moving $v_1$, so we have $u'\in U_1$ which completes a $K_{2,2}$.  Now we move $u_0$ to $U_1$ to finish.

\noindent
\textbf{Case 2.} $U_0\cup U_2^M= \emptyset$ and $|U_2|$ is odd.  Now there are no movable vertices in $U_1$ or $U_2$.  So choose $\ell_1, \ell_2$ such that $|U_1|=\ell_1s+1$, $|U_2|=\ell_2s-1$.  If it is not the case that $|V_0\cup V_1|\geq \ell_1s+2$ or $|V_0\cup V_2|\geq \ell_2s$, then $V_0=\emptyset$, $|V_1|=\ell_1s+1$, $|V_2|=\ell_2s-1$, and $V_1^M=\emptyset$.  Without loss of generality, suppose $|V_0\cup V_1|\geq \ell_1s+1$.  
Let $b:=|V_1\cup V_0|-|U_1|$.


\textbf{Case 2.1.} $b=0$.  Note that since $b=0$, $U_0=V_0=U_2^M=V_1^M=\emptyset$ for $i=1,2$.  We first show that if there is a vertex $u_i\in U_i$ such that $\deg(u_i, V_{3-i})\geq 2$, then we would be done.  Without loss of generality, suppose there exists $u_1\in U_1$ such that $\deg(u_1, V_2)\geq 2$.  Let $v,v'\in N(u_1)\cap V_2$.  Since $\delta(V_1, U_2)+\delta(U_2, V_1)\geq 1$, there is an edge $v_1u_2\in E(V_1, U_2)$.  Let $v_2\in V_2\cap N(u_2)\setminus\{v,v'\}$.  By \eqref{twocommon} we know that $v_1$ and $v_2$ have a common neighbor $u_0$ which is different than $u_2$.  If $u_0\in U_1$, then we have a copy of $K_{2,2}$ with one vertex in each of $U_1, U_2, V_1, V_2$ and we are done, so suppose $u_0\in U_2$.  Then we choose $u'\in (N(v)\cap N(v'))\cap (U_2\setminus\{u_0\})$.  Thus we can move $u$ and $v_1$ to finish.  So we may suppose that 
\begin{equation}\label{atmost1}
\Delta(U_1, V_2), \Delta(U_2, V_1)\leq 1.
\end{equation}
By (\ref{u*}), there is a vertex $u^*\in U$ such that $\deg(u^*, V)\geq \frac{n}{2}+2$.  Without loss of generality, suppose $u^*\in U_1$.  Then by (\ref{atmost1}) we have $|U_1|=|V_1|\geq \frac{n}{2}+1$, which in turn implies that $|U_2|=|V_2|\leq \frac{n}{2}-1$.  However, now we have $\delta(V_2, U_1)\geq 2$, and thus there exists $u\in U_1$ such that $\deg(u, V_2)\geq 2$, contradicting (\ref{atmost1}).

\textbf{Case 2.2.} $b\geq 1$.  Suppose first that $|V_1\setminus V_1^M|\geq \ell_1s+3$.  Let $b_1':=|V_1\setminus V_1^M|-(\ell_1s+2)$.   We have 
$$
\delta(V_1\setminus V_1^M, U_2)+\delta(U_2, V_1\setminus V_1^M)\geq n+1-(\ell_1s+1+\ell_2s-2-b_1')=b_1'+2.
$$
So we apply Lemma \ref{lemma:diagonalsum}(i) with $A=V_1\setminus V_1^M$ and $B=U_2$ to get a set of $b_1'$ vertex disjoint $s$ stars from $V_1\setminus V_1^M$ to $U_2$ and one $s$-star from $U_2$ to $V_1\setminus V_1^M$.

So we may suppose $|V_1\setminus V_1^M|\leq \ell_1s+2$.  Reset $V_1:=V_1\setminus V_1^M$ and $V_0:=V_0\cup V_1^M$, then partition $V_0=V_0^1\cup V_0^2$ so that $|V_1\cup V_0^1|=l_1s+2$ and $|V_2\cup V_0^2|=l_2s-2$.  We have
\begin{equation}\label{b=1}
\delta(V_1\cup V_0^1, U_2)+\delta(U_2, V_1\cup V_0^1)\geq n+1-(\ell_1s+1+\ell_2s-2)=2.
\end{equation}
We first observe that if $\delta(V_1\cup V_0^1, U_2)\geq 2$, then there will be a vertex $u_2\in U_2$ such that $\deg(u_2, V_1)\geq 2$ in which case we would be done, so suppose not.  This implies that $|U_1|\geq \frac{n}{2}$.

First assume that $|V_0^1|\leq 1$.  By (\ref{b=1}), one of $\delta(U_2, V_1\cup V_0^1)\geq 2$ or $\delta(V_1\cup V_0^1, U_2)\geq 1$ must hold.  Since $|V_1\cup V_0^1|>|U_2|$, in either case there is a vertex $u\in U_2$ such that $\deg(u, V_1\cup V_0^1)\geq 2$, in which case we are done since $|V_0^1|\leq 1$.

So suppose $|V_0^1|\geq 2$.  Now if $\delta(V_2\cup V_0^2, U_1)\geq 2$, then there will be a vertex $u_1\in U_1$ such that $\deg(u_1, V_2)\geq 2$ in which case we would be done, since we can also move two vertices from $V_0^2$, so suppose not.  This implies that $|U_2|\geq \frac{n}{2}$ and since $|U_1|\geq \frac{n}{2}$, we have $|U_1|=|U_2|=\frac{n}{2}$.  So let $v_2\in V_2$ with $\deg(v_2, U_1)=1$ and let $v_1\in N(u_1)\cap V_1$.  By \eqref{twocommon}, $v_1$ and $v_2$ have a common neighbor in $U_2$ (since $\deg(v_2, U_1)=1$) which completes a $K_{2,2}$.  We finish by moving one additional vertex from $V_0^1$ to $V_2$.

\subsubsection{Case $s\geq 3$}

The following proof has many cases, so we provide an outline for reference.

\begin{easylist} 

\ListProperties(Style*=\bfseries)

# $|V_1|\leq k_1s$ and $|V_0\cup V_1|\leq k_1s+r$
 
# $\exists \ell_1\geq k_1$, $\exists Y\subseteq V_1^M$ and $\exists V_0'\subseteq V_0$ such that $|(V_1\setminus Y)\cup V_0'|=\ell_1s$.  

## $|V_1|\leq k_1s$

### $|V_0\cup V_1|\geq k_1s+s$




### $|V_0\cup V_1|<k_1s+s$






## $|V_1|>k_1s$

### $|V_1\setminus V_1^M|\leq k_1s$

#### $|U_0\cup U_2|\geq k_2s$

#### $|U_0\cup U_2|<k_2s$ 

##### $|V_0\cup V_1|\geq k_1s+s$

###### $|U_0\cup U_1|\geq k_1s+s$

###### $|U_0\cup U_1|<k_1s+s$

##### $|V_0\cup V_1|<k_1s+s$

### $|V_1\setminus V_1^M|>k_1s$

#### $\exists \ell_1$, $\exists Y\subseteq V_1^M$ such that $|V_1\setminus Y|=\ell_1s$ 

##### $|U_0\cup U_2|< \ell_2s$ (i.e. $|U_1|>\ell_1s$)

##### $|U_0\cup U_2|\geq \ell_2s$

#### $\exists \ell_1$, $\exists V_0'\subseteq V_0$ such that $|V_1\cup V_0'|=\ell_1s$ 

##### $|U_0\cup U_2|<\ell_2 s$ 

##### $|U_0\cup U_2|\geq \ell_2 s$

%
%
%

%
%
%

# For some $\ell_1\geq k_1$ we have $\ell_1s<|V_1\setminus V_1^M|\leq |V_1\cup V_0|<\ell_1s+s$

## $|U_2\setminus U_2^M|\geq \ell_2s$

## $|U_2\setminus U_2^M|<\ell_2s$

### $|U_0\cup U_1|\geq \ell_1s+s$

#### $|U_1|\leq \ell_1s$

#### $|U_1|>\ell_1s$

##### $\ell_1>k_1$

##### $\ell_1=k_1$



%
%
%

%
%
%

### $\ell_1s<|U_0\cup U_1|<\ell_1s+s$

#### $|U_1|\leq \ell_1s$
%
%
%
%
%


#### $|U_1|>\ell_1s$






##### For some $i\in\{1,2\}$ we have $\delta(V_i, U_{3-i})\geq s$ or $\delta(U_{3-i}, V_i)\geq s$

##### For all $i\in\{1,2\}$ we have $\delta(V_i, U_{3-i})<s$ and $\delta(U_{3-i}, V_i)<s$

\end{easylist}

Recall the following definitions. For $i=1,2$, $U_i^M=\{u\in U_i:\deg(u, V_{3-i})>\alpha^{1/3}n\}$ and $V_i^M=\{v\in V_i: \deg(v, U_{3-i})>\alpha^{1/3}n \}$. Also recall $U_1^M=\emptyset=V_2^M$ by Claim \ref{preprocess}.

\noindent
\textbf{Case 1} $|V_1|\leq k_1s$ and $|V_0\cup V_1|\leq k_1s+r$.  Let $b_2:=|V_2|-k_2s$ and note that $b_2\geq -r$.  We have 
\begin{equation} 
\label{eq1.1} \delta(U_1, V_2)\geq k_1s+s+r-(k_1s-b_2)\geq s+r+b_2\geq s. 
\end{equation}

\begin{claim}\label{Claim1}
If $|V_0\cup V_1|\geq k_1s$, then there exists $V_0'\subseteq V_0$ such that $|V_1\cup (V_0\setminus V_0')|=k_1s$.  If $|V_0\cup V_1|<k_1s$, then there exists a set of vertex disjoint $s$-stars with centers $C\subseteq V_2$  and leaves in $U_1$ such that $|V_0\cup V_1|+|C|=k_1s$.
\end{claim}

\begin{proof}
If $|V_0\cup V_1|\geq k_1s$, we just choose $V_0'\subseteq V_0$ such that $|V_1\cup (V_0\setminus V_0')|=k_1s$.  Otherwise $b_2\geq 0$ and thus by \eqref{eq1.1} and $\Delta(V_2, U_1)<2\alpha^{1/3}k_2s$, we can apply Lemma \ref{STARSlemma}(ii) to get a set of $b_2$ vertex disjoint $s$-stars from $V_2$ to $U_1$ with centers $C$.  So we have $|V_0\cup V_1\cup C|=k_1s$.
\end{proof}

Let $a_2:=|U_2|-k_2s$. We have two cases.

Suppose $a_2\geq 0$. Claim \ref{k2approxk1} gives $\delta(V_1, U_2)\geq k_2s+2s-5-r-(k_1s-a_2)\geq s+a_2$.  So by Lemma \ref{STARSlemma}(ii) there are $a_2$ vertex disjoint $s$-stars from $U_2$ to $V_1$ with centers $C_U$.  So we can make $|U_0\cup U_1\cup C_U|=k_1s$ and apply Claim \ref{Claim1} to finish.

Suppose $a_2<0$. Then $|U_0\cup U_1|>k_1s$.  If $|U_1|\leq k_1s$, then there exists $U_0'\subseteq U_0$ such that $|U_1\cup (U_0\setminus U_0')|=k_1s$ and we apply Claim \ref{Claim1} to finish.  Otherwise $|U_1|>k_1s$ and let $a_1:=|U_1|-k_1s>0$.  If $b_2>0$, then we have $$\delta(U_1, V_2)+\delta(V_2, U_1)\geq 3s-5+a_1+b_2,$$ and we use Lemma \ref{lemma:diagonalsum}(i) to get a set of $a_1$ vertex disjoint $s$-stars from $U_1$ to $V_2$ with centers $C_U$ and a set of $b_2$ vertex disjoint $s$-stars from $V_2$ to $U_1$ with centers $C_V$.  Thus $|U_1\setminus C_U|=k_1s$ and $|V_0\cup V_1\cup C_V|=k_1s$.  Finally suppose $b_2\leq 0$, i.e. $|V_0\cup V_1|\geq k_1s$.  If there exists a set of $a_1$ vertex disjoint $s$-stars from $U_1$ to $V_2$, then we can apply Claim \ref{Claim1} to finish.  We show that such a set exists.  We have 
\begin{equation} 
\label{eq1.2} \delta(V_2, U_1)\geq k_2s+2s-5-r-(k_2s-a_1)=2s-5-r+a_1\geq s-4+a_1. 
\end{equation} 
If $a_1\leq 3$, we use \eqref{eq1.1} and Lemma \ref{STARSlemma}(i) with $x=0$ to get a set of $a_1$ vertex disjoint $s$-stars from $U_1$ to $V_2$ with centers $C_U$.  Otherwise $a_1\geq 4$ and we use \eqref{eq1.2} and Lemma \ref{STARSlemma}(iii) or (v) to get a set of $a_1$ vertex disjoint $s$-stars from $U_1$ to $V_2$ with centers $C_U$.

\noindent
\textbf{Case 2.} There exists $\ell_1\geq k_1$, $Y\subseteq V_1^M$ and $V_0'\subseteq V_0$ such that $|(V_1\setminus Y)\cup V_0'|=\ell_1s$.  Let $\ell_1\geq k_1$ be minimal.

\textbf{Case 2.1.}
$|V_1|\leq k_1s$.  By Case 1 we have $|V_0\cup V_1|>k_1s+r$.  This implies that there exists $V_0'\subseteq V_0$ such that $|V_1\cup V_0'|=k_1s$ and $|(V_0\cup V_2)\setminus V_0'|=k_2s$.  We now try to make $|U_1|=k_1s$ or $|U_2|=k_2s$.  Reset $U_2:=U_2\setminus U_2^M$ and $U_0:=U_0\cup U_2^M$.  Let $a_1:=|U_1|-k_1s$ and $a_2:=|U_2|-(k_2s-s)$. We have 
\begin{equation}\label{eq2.1.a}
\delta(V_2, U_1)\geq k_2s+2s-5-r-(k_2s-a_1)=2s-5-r+a_1
\end{equation}
and
\begin{equation}\label{eq2.1.b}
\delta(V_1, U_2)\geq k_2s+2s-5-r-(k_1s+s-a_2)=(k_2-k_1)s+s-5-r+a_2.
\end{equation}
If $|U_2|\geq k_2s$ i.e. $a_2\geq s$, then by \eqref{eq2.1.b} and Claim \ref{k2approxk1} we have $\delta(V_1, U_2)\geq s-1+(a_2-s)$ and thus Lemma \ref{STARSlemma}(ii) gives $a_2-s$ vertex disjoint $s$-stars from $U_2$ to $V_1$ with centers $C_U$ such that $|U_2\setminus C_U|=k_2s$.  Otherwise we have $|U_0\cup U_1|>k_1s$.  If $|U_1|\leq k_1s$, then we choose $U_0'\subseteq U_0$ such that $|U_1\cup (U_0\setminus U_0')|=k_1s$.  So suppose $|U_1|>k_1s$, i.e. $a_1>0$.  

\textbf{Case 2.1.1.} $|V_0\cup V_1|\geq k_1s+s$.  If $|U_0\cup U_1|\geq k_1s+s$, then we are done: either $a_1\leq s$ and we just choose $U_0'\subseteq U_0$ and $V_0'\subseteq V_0$ such that $|V_1\cup (V_0\setminus V_0')|=k_1s+s$ and $|U_1\cup (U_0\setminus U_0')|=k_1s+s$ or else $a_1> s$ and thus \eqref{eq2.1.a} gives $\delta(V_2, U_1)\geq 2s-4+(a_1-s)\geq s-1+(a_1-s)$ and thus Lemma \ref{STARSlemma}(ii) allows us to find $a_1-s$ vertex disjoint $s$-stars from $U_1$ to $V_2$.  So suppose $|U_0\cup U_1|<k_1s+s$ and thus $a_2>0$.  

$k_2=k_1$.  By Claim \ref{k2approxk1}, $r\leq \frac{s-6}{2}$ which implies $\delta(V_2, U_1)\geq s-1+a_1$ by \eqref{eq2.1.a}.  So there are $a_1$ vertex disjoint $s$-stars from $U_1$ to $V_2$ by Lemma \ref{STARSlemma}(ii).

$k_2=k_1+1$.  By Claim \ref{k2approxk1}, $r\leq s-3$ which implies $\delta(V_2, U_1)\geq s-2+a_1$ by (\ref{eq2.1.a}).  If $a_1\geq 2$ or $r\leq s-4$, then there are $a_1$ vertex disjoint $s$-stars from $U_1$ to $V_2$ by Lemma \ref{STARSlemma}(iii), so suppose $a_1=1$ and $r=s-3$.  Furthermore we have $\delta(V_1, U_2)\geq s-2+a_2$ by (\ref{eq2.1.b}).  If $a_2\geq 2$, then there are $a_2$ vertex disjoint $s$-stars from $U_2$ to $V_1$ by Lemma \ref{STARSlemma}(iii), so suppose $a_2=1$.  Note that we would be done unless $\Delta(U_1, V_2)\leq s-1$ and $\Delta(U_2, V_1)\leq s-1$.  Let $d_1:=k_1s-|V_1|$ and let $d_2:=k_2s-|V_2|$.  Note that $|V_0|=d_1+d_2\geq s$.  
Let $\hat{U_1}=\{u\in U_1:\deg(u, V_1)\leq k_1s-d_1-4\}$ and suppose that $\hat{U_1}\neq \emptyset$.  
So we have $$\delta(\hat{U_1}, V_0)+\delta(U_2, V_0)\geq 2(k_1s+s+r)-(k_1s-d_1-4+s-1)-(k_2s-d_2+s-1)\geq |V_0|+s.$$
This implies that we can find a $K_{s,s}$ with one vertex in $U_1$, $s-1$ vertices in $U_2$ and $s$ vertices in $V_0$.  So we may suppose that $\hat{U_1}=\emptyset$.  Note that $\delta(U_1, V_1)\geq k_1s-d_1-3=|V_1|-3$.  Since $\delta(V_1, U_2)\geq s-1$, there exists a set of $3s-2$ vertex disjoint $(s-1)$-stars from $U_2$ to $V_1$ with centers $C_U$.  Let $v_2\in N(C_U)\cap V_2$. Since $\delta(V_2, U_1)\geq s-1$, we can let $L_U\subseteq N(v_2)\cap U_1$ such that $|L_U|=s-1$.  Since $\delta(U_1, V_1)\geq |V_1|-3$, the leaves of at least one of the $(s-1)$-stars from $U_2$ to $V_1$ forms a $K_{s-1,s-1}$ with $L_U$.  This allows us to move a vertex $u_2\in U_2$ to $U_1$ and $v_2$ to $V_1$. This makes $|U_2\setminus \{u_2\}|=k_2s-s$, and we choose $V_0'\subseteq V_0$ such that $|V_0'\cup V_2\setminus\{v_2\}|=k_2s-s$.

$k_2\geq k_1+2$.  In this case, we see from \eqref{eq2.1.b} that $\delta(V_1, U_2)\geq 2s-4+a_2\geq s-1+a_2$.  So there are $a_2$ vertex disjoint $s$-stars from $U_2$ to $V_1$ by Lemma \ref{STARSlemma}(ii).  Then we choose $V_0'\subseteq V_0$ such that $|V_1\cup (V_0\setminus V_0')|=k_1s+s$.

\textbf{Case 2.1.2.} $|V_0\cup V_1|<k_1s+s$. Let $b_2:=|V_2|-(k_2s-s)$ and note that $b_2>0$.

$k_2=k_1$.  Then $r\leq \frac{s-6}{2}$ which implies $\delta(V_2, U_1)\geq s-1+a_1$ by (\ref{eq2.1.a}).  So by Lemma \ref{STARSlemma}(ii) there are $a_1$ vertex disjoint $s$-stars from $U_1$ to $V_2$. 

$k_2=k_1+1$.  Then $r\leq s-3$ which implies $\delta(V_2, U_1)\geq s-2+a_1$ by (\ref{eq2.1.a}).  If $a_1\geq 2$, then there are $a_1$ vertex disjoint $s$-stars from $U_1$ to $V_2$, so suppose $a_1=1$.  We have $|V_2|=k_2s-s+b_2=k_1s+b_2$.  If $b_2\geq 2$, then $|V_2|>|U_1|$ which together with $\delta(V_2, U_1)\geq s-1$ implies that there is a vertex in $U_1$ with at least $s$ neighbors in $V_2$, in which case we are done.  So suppose $b_1=1$ and thus $|V_2|=|U_1|$.  So if there is a vertex in $V_2$ with $s$ neighbors in $U_1$, then there is a vertex in $U_1$ with $s$ neighbors in $V_2$, so suppose not.  Together with $\delta(V_2, U_1)\geq s-1$, this implies that $G[U_1, V_2]$ is $(s-1)$-regular.  So we have $\delta(V_2, U_0\cup U_2)\geq k_2s+2s-5-r-(s-1)\geq k_2s-1=|U_0\cup U_2|$ which implies that $G[V_2, U_0\cup U_2]$ is complete, and thus we can choose a vertex $u_1\in U_1$ and a vertex $v_1\in N(u_1)\cap V_1$.  Since $\deg(u_1, V_2)=s-1$ and $\deg(v_1, U_0\cup U_2)\geq k_2s+2s-5-r-(k_1s+1)\geq 2s-3\geq s$ we can move $u_1$ and $v_1$.  Then we replace $v_1$ with a vertex from $V_0$ as $V_0\neq \emptyset$.

$k_2\geq k_1+2$. 

\begin{claim}\label{Claim 2123}
If $|U_0\cup U_1|\geq k_1s+s$ and $|U_1|\leq k_1s+s$, then there exists $U_0'\subseteq U_0$ such that $|(U_0\cup U_1)\setminus U_0'|=k_1s+s$. If $|U_0\cup U_1|< k_1s+s$, then there exists a set of vertex disjoint $s$-stars with centers $C\subseteq U_2$ and leaves in $V_1$ such that $|U_0\cup U_1|+|C|=k_1s+s$.
\end{claim}

\begin{proof}
Suppose first that $|U_0\cup U_1|\geq k_1s+s$ and $|U_1|\leq k_1s+s$.  Let $U_0'\subseteq U_0$ so that $|(U_0\cup U_1)\setminus U_0'|=k_1s+s$.  
Now suppose $|U_0\cup U_1|< k_1s+s$ and let $a_2:=|U_2|-(k_2s-s)$.  Since $k_2\geq k_1+2$, \eqref{eq2.1.b} gives $\delta(V_1, U_2)\geq 2s-4+a_2\geq s-1+a_2$ and thus by Lemma \ref{STARSlemma}(ii) there is a set of $a_2$ vertex disjoint $s$-stars with centers $C\subseteq U_2$ and leaves in $V_1$ such that $|U_0\cup U_2|+|C|=k_1s+s$.
\end{proof}

We have 
\begin{equation}
\delta(U_1, V_2)\geq k_1s+s+r-(k_1s+s-b_2)=r+b_2.
\end{equation}
If $r\geq s-b_2$, then $\delta(U_1, V_2)\geq s$ and we apply Lemma \ref{STARSlemma}(iii) to get a set of $b_2$ vertex disjoint $s$-stars from $V_2$ to $U_1$.  So suppose $r\leq s-1-b_2$.  By \eqref{eq2.1.a} we have 
\begin{equation}
\delta(V_2, U_1)\geq s-4+a_1+b_2.\label{2123}  
\end{equation}
We would be done unless $2\leq a_1+b_2\leq 3$. Note also that we have 
\begin{equation}
\delta(V_1, U_0\cup U_2)\geq k_2s+2s-5-r-(k_1s+a_1)\geq (k_2-k_1)s+s-4+b_2-a_1\geq 3s-4+b_2-a_1.\label{2123i}
\end{equation}

First suppose $b_2=2$ and $a_1=1$.  By \eqref{2123} we have $\delta(V_2, U_1)\geq s-1$, and since $|V_2|>|U_1|$ there exists some $u\in U_1$ such that $\deg(u, V_2)\geq s$.  Thus we can move one vertex from $U_1$.

Now suppose $b_2=1$.  If there is a vertex $v_2\in V_2$ such that $\deg(v_2, U_1)\geq s$, then $|(V_0\cup V_1)\cup\{v_2\}|=k_1s+s$ and we apply Claim \ref{Claim 2123} to finish.  So suppose $\Delta(V_2, U_1)\leq s-1$.  

If $a_1=2$, we have $\delta(V_2, U_0\cup U_2)\geq k_2s+2s-5-r-(s-1)\geq k_2s-2=|U_0\cup U_2|$ which implies that $G[V_2, U_0\cup U_2]$ is complete.  Since $\delta(V_2, U_1)\geq s-1$ and $|V_2|>|U_1|$, there is a vertex $u_1\in U_1$ such that $\deg(u_1, V_2)\geq s$ and since $\delta(V_2, U_1)\geq s-1$ and $\Delta(U_1, V_2)<2\alpha^{1/3}k_1s$, there is another vertex $u_1'\in U_1$ such that $\deg(u_1', V_2)\geq s-1$ and the neighborhoods of $u_1$ and $u_1'$ in $V_2$ are disjoint.  Let $v_1'\in N(u_1')\cap V_1$; by \eqref{2123i} $\deg(v_1', U_0\cup U_2)\geq s-1$ and thus since $G[V_2, U_0\cup U_2]$ is complete we can move $u_1$, $u_1'$ to make $|U_1|=k_1s$.

If $a_1=1$, we have $\delta(V_2, U_0\cup U_2)\geq k_2s+2s-5-r-(s-1)\geq k_2s-2=|U_0\cup U_2|-1$.  Since $\delta(V_2, U_1)\geq s-2$ and $|V_2|>|U_1|$, there is a vertex $u_1\in U_1$ such that $\deg(u_1, V_2)\geq s-1$.  Let $v_1\in V_1\cap N(u_1)$; by \eqref{2123i} we have $\deg(v_1, U_0\cup U_2)\geq 3s-4\geq 2s-1$.  Since $\delta(V_2, U_0\cup U_2)\geq |U_0\cup U_2|-1$, $K_{s-1,s-1}\subseteq G[N(u_1)\cap V_2, N(v_1)\cap (U_0\cup U_2)]$.  Thus we can move $u_1$.

\textbf{Case 2.2} $|V_1|>k_1s$.

\textbf{Case 2.2.1.} $|V_1\setminus V_1^M|\leq k_1s$.  Let $Y\subseteq V_1^M$ such that $|V_1\setminus Y|=k_1s$.

\textbf{Case 2.2.1.1.} $|U_0\cup U_2|\geq k_2s$. If $|U_2|\leq k_2s$, then there exists $U_0'\subseteq U_0$ such that $|U_1\cup U_0'|=k_1s=|V_1\setminus Y|$ and we are done.  If not, then we have $|U_2|>k_2s$. So let $a_2:=|U_2|-k_2s$.  We have $\delta(V_1, U_2)\geq k_2s+2s-5-r-(k_1s-a_2)=(k_2-k_1)s+2s-5-r+a_2\geq s-1+a_2$ by Claim \ref{k2approxk1}, and thus we can apply Lemma \ref{STARSlemma}(ii) to get a set of $a_2$ vertex disjoint $s$-stars from $U_2$ to $V_1$.  Since $|(V_0\cup V_2)\cup Y|=k_2s$, we are done.

\textbf{Case 2.2.1.2.} $|U_0\cup U_2|<k_2s$.  Set $a_1:=|U_1|-k_1s$ and note that $a_1\geq 1$.  We have \begin{equation}\label{V2U1-2212} \delta(V_2, U_1)\geq k_2s+2s-5-r-(k_2s-a_1)=2s-5-r+a_1.\end{equation}

\textbf{Case 2.2.1.2.1.} $|V_0\cup V_1|\geq k_1s+s$.

\textbf{Case 2.2.1.2.1.1.} $|U_0\cup U_1|\geq k_1s+s$. If $a_1\leq s$, we can choose $U_0'\subseteq U_0$ and $Y'\subseteq V_1^M\cup V_0$ so that $|(U_0\cup U_1)\setminus U_0'|=|(V_0\cup V_1)\setminus Y'|=k_1s+s$.  If $a_1>s$, then \eqref{V2U1-2212} implies $\delta(V_2, U_1)\geq 2s-4+(a_1-s)\geq s-1+(a_1-s)$ and thus we can apply Lemma \ref{STARSlemma}(ii) to get $a_1-s$ vertex disjoint $s$-stars from $U_1$ to $V_2$.  Now let $Y'\subseteq V_1^M\cup V_0$ so that $|U_1|-(a_1-s)=|(V_0\cup V_1)\setminus Y'|=k_1s+s$.

\textbf{Case 2.2.1.2.1.2.} $|U_0\cup U_1|<k_1s+s$.  Let $a_2=|U_2|-(k_2s-s)$.  We have 
\begin{equation}\label{V1U2-22122}
\delta(V_1, U_2)\geq k_2s+2s-5-r-(k_1s+s-a_2)=(k_2-k_1)s+s-5-r+a_2.
\end{equation}  

If $k_2=k_1$, then $r\leq \frac{s-6}{2}$.  By \eqref{V2U1-2212} we have $\delta(V_2, U_1)\geq \frac{3s-4}{2}+a_1\geq s-1+a_1$.  So by Lemma \ref{STARSlemma}(ii), we can move $a_1$ vertices from $U_1$ so that $|U_1|-a_1=k_1s=|V_1\setminus Y|$.

If $k_2=k_1+1$, then $r\leq s-3$. By \eqref{V1U2-22122} and \eqref{V2U1-2212} we have $\delta(V_1, U_2)\geq s-2+a_2$ and $\delta(V_2, U_1)\geq s-2+a_1$.  We would be done if either $\delta(V_1, U_2)\geq s$ or $\delta(V_2, U_1)\geq s$, because $|V_0\cup V_1|\geq k_1s+s$ and $|V_1\setminus V_1^M|\leq k_1s$.  So we may suppose $a_1=a_2=1$ and $r=s-3$.  We have $|V_1|\geq |U_2|$, $\delta(V_1, U_2)\geq s-1$, and at least one vertex $v_1\in V_1^M$ such that $\deg(v_1, U_2)\geq \alpha^{1/3}n$. Thus there is a vertex $u_2\in U_2$ such that $\deg(u_2, V_1)\geq s$.  So we have $|(U_0\cup U_2)\cup \{u_2\}|=k_1s+s$ and $|V_0\cup V_1|\geq k_1s+s$ with $|V_1\setminus V_1^M|\leq k_1s$ so we are done.

Finally, suppose that $k_2\geq k_1+2$.  We have $\delta(V_1, U_2)\geq (k_2-k_1)s+s-5-r+a_2\geq 2s-4+a_2\geq s-1+a_2$ since $s\geq 3$.  Thus we can find $a_2$ vertex disjoint $s$-stars from $U_2$ to $V_1$ by Lemma \ref{STARSlemma}(ii) and we have $|(U_0\cup U_1)|+a_2=k_1s+s$.  Since $|V_0\cup V_1|\geq k_1s+s$ and $|V_1\setminus V_1^M|\leq k_1s$ we are done.

\textbf{Case 2.2.1.2.2.} $|V_0\cup V_1|<k_1s+s$.  Set $b_2:=|V_2|-(k_2s-s)$ and $b_1:=|V_1|-k_1s$.  Note that  
$1\leq b_1, b_2\leq s-1$.

If $k_2=k_1$, then $r\leq \frac{s-6}{2}$ by Claim \ref{k2approxk1}. So by \eqref{V2U1-2212} we have $\delta(V_2, U_1)\geq \frac{3s-4}{2}+a_1\geq s-1+a_1$.  By Lemma \ref{STARSlemma}(ii), we can move $a_1$ vertices from $U_1$ so that $|U_1|-a_1=k_1s=|V_1\setminus Y|$.

If $k_2=k_1+1$, then $r\leq s-3$ and by \eqref{V2U1-2212} we have \begin{equation}\label{V2U1-221212}\delta(V_2, U_1)\geq s-2+a_1.\end{equation}  If $a_1\geq 2$ or $r\leq s-4$, then \eqref{V2U1-221212} gives $\delta(V_2, U_1)\geq s-2+a_1\geq s$ in which case we can apply Lemma \ref{STARSlemma}(iii) to get a set of $a_1$ vertex disjoint $s$-stars from $U_1$ to $V_2$. So suppose $a_1=1$ and $r=s-3$.  We have $\delta(U_1, V_2)\geq k_1s+s+r-(k_1s+s-b_2)=r+b_2\geq s-3+b_2$.  If $b_2\geq 3$, then we have $\delta(U_1, V_2)\geq s$ and thus we can move a single vertex from $U_1$ to make $|U_1|-a_1=k_1s=|V_1\setminus Y|$.  So suppose $1\leq b_2\leq 2$.  By \eqref{V2U1-221212}, we have $\delta(V_2, U_1)\geq s-1$. If $b_2=2$, then $|V_2|=k_1s+2>k_1s+1=|U_1|$ and since $\delta(V_2, U_1)\geq s-1$ there exists $u\in U_1$ such that $\deg(u, V_2)\geq s$. So we move $u$ to $U_2$ and $|U_1\setminus\{u\}|=k_1s=|V_1\setminus Y|$.  So we may suppose that $b_2=1$.  Since $\delta(V_2, U_1)\geq s-1$, if there was a vertex $v\in V_2$ such that $\deg(v, U_1)\geq s$, then there exists $u\in U_1$ such that $\deg(u, V_2)\geq s$ in which case we would be done.  So we can suppose $\Delta(U_1, V_2), \Delta(V_2, U_1)\leq s-1$.  Then since $\delta(V_2, U_1)\geq s-1$ by \eqref{V2U1-221212}, we have that $G[U_1, V_2]$ is $(s-1)$-regular.  So we have $\delta(V_2, U_0\cup U_2)\geq k_2s+2s-5-r-(s-1)\geq k_2s-1=|U_0\cup U_2|$ and thus $G[V_2, U_0\cup U_2]$ is complete.  Since $|V_1|=k_1s+1$ and $|V_1\setminus V_1^M|\leq k_1s$, there exists some $v_1\in V_1$ with $\deg(v_1, U_2)> \alpha^{1/3}n$.  Let $u_1\in U_1\cap N(v_1)$.  Since $\deg(u_1, V_2)=s-1$ and $G[V_2, U_0\cup U_2]$ is complete there is a copy of $K_{s,s}$ which contains $u_1$ and $v_1$.  Thus $|U_1\setminus\{u_1\}|=k_1s=|V_1\setminus Y|$.

Finally, suppose $k_2\geq k_1+2$.  We first prove the following claim.

\begin{claim}\label{Claim 22122}
If $|U_0\cup U_1|\geq k_1s+s$ and $|U_1|\leq k_1s+s$, then there exists $U_0'\subseteq U_0$ such that $|(U_0\cup U_1)\setminus U_0'|=k_1s+s$. If $|U_0\cup U_1|< k_1s+s$, then there exists a set of vertex disjoint $s$-stars with centers $C\subseteq U_2$ and leaves in $V_1$ such that $|U_0\cup U_1|+|C|=k_1s+s$.
\end{claim}

\begin{proof}
Suppose first that $|U_0\cup U_1|\geq k_1s+s$ and $|U_1|\leq k_1s+s$.  Let $U_0'\subseteq U_0$ so that $|(U_0\cup U_1)\setminus U_0'|=k_1s+s$.  
Now suppose $|U_0\cup U_1|< k_1s+s$ and let $a_2:=|U_2|-(k_2s-s)$.  Equation \eqref{V1U2-22122} holds in this case.  Since $k_2\geq k_1+2$, \eqref{V1U2-22122} gives $\delta(V_1, U_2)\geq 2s-4+a_2\geq s-1+a_2$ and thus by Lemma \ref{STARSlemma}(ii) there is a set of $a_2$ vertex disjoint $s$-stars with centers $C\subseteq U_2$ and leaves in $V_1$ such that $|U_0\cup U_2|+a_2=k_1s+s$.
\end{proof}

We have 
\begin{equation}
\delta(U_1, V_2)\geq k_1s+s+r-(k_1s+s-b_2)=r+b_2.\label{221:rb2}
\end{equation}
If $r\geq s-b_2$, then $\delta(U_1, V_2)\geq s$ and we can apply Lemma \ref{STARSlemma}(iii) to get a set of $a_1$ vertex disjoint $s$-stars from $U_1$ to $V_2$ giving $|U_1|-a_1=k_1s=|V_1\setminus Y|$.  So suppose $r\leq s-1-b_2$.  By \eqref{V2U1-2212} we have 
\begin{equation}
\delta(V_2, U_1)\geq s-4+a_1+b_2.\label{221:a1b2}
\end{equation}  
If $\delta(V_2, U_1)\geq s$, we would be done by moving $a_1$ vertices from $U_1$, so suppose $2\leq a_1+b_2\leq 3$.

If $b_2=2$ and $a_1=1$, then $\delta(V_2, U_1)\geq s-1$ and since $|V_2|>|U_1|$, there is a vertex $u\in U_1$ with $\deg(u, V_2)\geq s$, which we can move $|U_1|-a_1=k_1s=|V_1\setminus Y|$.

If $a_1=2$ and $b_2=1$, then $\delta(V_2, U_1)\geq s-1$ by \eqref{221:a1b2}.  If $r\leq s-3$, then \eqref{V2U1-2212} would give $\delta(V_2, U_1)\geq s$ in which case we would be done by moving two vertices from $U_1$, so suppose $r=s-2$. If there is a vertex $v_2\in V_2$ with $\deg(v_2, U_1)\geq s$, we can move $v_2$ so that $|(V_0\cup V_2) \cup \{v_2\}|=k_1s+s$ and apply Claim \ref{Claim 22122} to finish.  So suppose $\Delta(V_2, U_1)\leq s-1$. So for all $v\in V_2$, $\deg(v, U_0\cup U_2)\geq k_2s+2s-5-r-(s-1)=k_2s-2=|U_0\cup U_2|$, which implies $G[V_2, U_0\cup U_2]$ is complete.  
Since $|V_2|>|U_1|$ and $\delta(V_2, U_1)\geq s-1$, there is a vertex $u_1\in U_1$ with $\deg(u_1, V_2)\geq s$.  Let $L$ be a subset of $N(u_1)\cap V_2$ of size $s$.  Let $v_1\in V_1^M$ and note that $\delta(U_1, V_2)\geq s-1$ by \eqref{221:rb2} and the fact that $r=s-2$.  Since $\Delta(V_2, U_1)\leq s-1$ there must be a vertex $u_1'\in U_1\cap N(v_1)$ such that $\deg(u_1', V_2\setminus L)\geq s-1$.   Then since $G[V_2, U_0\cup U_2]$ is complete, $u_1$ and $v_1$ are contained in a copy of $K_{s,s}$.  Thus $|U_1\setminus\{u_1, u_1'\}|=k_1s=|V_1\setminus Y|$.

Now in the final case we have $a_1=1=b_2$.  If there were a vertex $v_2\in V_2$ such that $\deg(v_2, U_1)\geq s$, then $|(V_0\cup V_1)\cup \{v_2\}|=k_1s+s$ and we apply Claim \ref{Claim 22122} to finish.  So suppose $\Delta(V_2, U_1)\leq s-1$.  Since $r\leq s-2$, we have $\delta(V_2, U_0\cup U_2)\geq k_2s+2s-5-r-(s-1)\geq k_2s-2=|U_0\cup U_2|-1$.  Also $\delta(V_1, U_0\cup U_2)\geq k_2s+2s-5-r-(k_1s+1)\geq(k_2-k_1)s+s-4\geq 3s-4\geq 2s-2$.  Since $\delta(V_2, U_1)\geq s-2$ and $|V_2|>|U_1|$, there exists $u_1\in U_1$ with $\deg(u_1, V_2)\geq s-1$.  Let $v_1\in N(u_1)\cap V_1$.  Since $v_1$ has $2s-2$ neighbors in $U_0\cup U_2$ and $\delta(V_2, U_0\cup U_2)\geq |U_0\cup U_2|-1$ there is a copy of $K_{s,s}$ which contains $u_1$ and $v_1$ with $s-1$ vertices in $U_0\cup U_2$ and $s-1$ vertices in $V_2$.  If $v_1\in V_1^M$, then $|U_1\setminus\{u_1\}|=k_1s=|V_1\setminus Y|$. If $v_1\notin V_1^M$, then let $Y'\subseteq Y$ with $|Y'|=|Y|-1$ and thus  $|U_1\setminus\{u_1\}|=k_1s=|(V_1\setminus\{v_1\})\setminus Y'|$.

\textbf{Case 2.2.2.} $|V_1\setminus V_1^M|>k_1s$.

\textbf{Case 2.2.2.1.} $\exists \ell_1$, $\exists Y\subseteq V_1^M$ such that $|V_1\setminus Y|=\ell_1s$.  Choose $\ell_1$ minimal and note that $\ell_1>k_1$ by Case 2.2.2.  Let $\ell_2:=m-\ell_1$.

\textbf{Case 2.2.2.1.1.} $|U_0\cup U_2|< \ell_2s$.  Let $a_1:=|U_1|-\ell_1s$. We have $\delta(V_2, U_1)\geq k_2s+2s-5-r-(\ell_2s-a_1)=(k_2-\ell_2)s+2s-5-r+a_1\geq 2s-4+a_1\geq s-1+a_1$, and thus we can find a set of $a_1$ vertex disjoint $s$-stars from $U_1$ to $V_2$.  This gives $|U_1|-a_1=\ell_1s=|V_1\setminus Y|$.  

\textbf{Case 2.2.2.1.2.} $|U_0\cup U_2|\geq \ell_2s$.  If $|U_2|\leq \ell_2s$, then there exists $U_0'\subseteq U_0$ such that $|U_1\cup U_0'|=\ell_1s=|V_1\setminus Y|$.  Otherwise $|U_2|>\ell_2s$.  Set $a_2:=|U_2|-\ell_2s$.  

We have $|V_1\setminus Y|=\ell_1s$ and since $\ell_1>k_1$ and $\ell_1$ is minimal, we have $|V_1^M\setminus Y|<s$.  Set $b_1:=|V_1\setminus V_1^M|-(\ell_1s-s)$.  We have 
\begin{equation}\label{V1-Y}
\delta(V_1\setminus Y, U_2)+\delta(U_2, V_1\setminus Y)\geq n+3s-5-(\ell_1s-a_2+\ell_2s)=3s-5+a_2
\end{equation}
and 
\begin{equation}\label{V1-V1M}
\delta(V_1\setminus V_1^M, U_2)+\delta(U_2, V_1\setminus V_1^M)\geq n+3s-5-(\ell_1s-a_2+\ell_2s+s-b_1)=2s-5+b_1+a_2.  
\end{equation}
If $\delta(V_1\setminus Y, U_2)\geq s$, then there are $a_2$ vertex disjoint $s$-stars from $U_2$ to $V_1$ by Lemma \ref{STARSlemma}(iii) and we are done.  Otherwise by \eqref{V1-Y} we have $\delta(U_2, V_1\setminus Y)\geq 2s-4+a_2\geq s$.  If $\delta(U_2, V_1\setminus V_1^M)\geq s$, then since $\Delta(V_1\setminus V_1^M, U_2)<\alpha^{1/3}n$ we can apply Lemma \ref{STARSlemma}(iii) to get a set of $a_2$ vertex disjoint $s$-stars from $U_2$ to $V_1$.  Likewise if $\delta(V_1\setminus V_1^M, U_2)\geq s$.  These two facts, together with \eqref{V1-V1M} imply $2\leq a_2+b_1\leq 3$.  If $a_2=1$, then since $\delta(U_2, V_1\setminus Y)\geq 2s-3\geq s$ and we only need to move one vertex, we are done.  So we only need to deal with the case when $a_2=2$, $b_1=1$, and $\delta(U_2, V_1\setminus V_1^M)=s-1=\delta(V_1\setminus V_1^M, U_2)$. Since $b_1=1$ we have $|V_1^M\setminus Y|=s-1$. If there exists a vertex $u_2\in U_2$ such that $\deg(u_2, V_1\setminus V_1^M)\geq s$, then since $\delta(U_2, V_1\setminus Y)\geq s$, we either have another vertex disjoint $s$-star and we are done, or every vertex in $U_2$ must have a neighbor in $N(u_2)\cap (V_1\setminus V_1^M)$.  However this implies that some vertex in $v'\in N(u_2)\cap (V_1\setminus V_1^M)$ has $\deg(v', U_2)>\alpha^{1/3}n$ contradicting the fact that vertices in $V_1\setminus V_1^M$ are not movable.  So we have $\Delta(U_2, V_1\setminus V_1^M)\leq s-1$.  Since $\delta(U_2, V_1\setminus Y)\geq 2s-4+a_2=2s-2$, $\Delta(U_2, V_1\setminus V_1^M)\leq s-1$ and $|V_1^M\setminus Y|=s-1$, every vertex in $U_2$ is adjacent to every vertex in $V_1^M\setminus Y$.  Since $\delta(V_1\setminus V_1^M, U_2)=s-1$, we can choose $v_1\in V_1\setminus V_1^M$ and $u_2, u_2'\in N(v_1)\cap U_2$.  Thus $\{v_1\}\cup (V_1^M\setminus Y)$ and $\{u_2, u_2'\}$ form a $K_{s,2}$ and thus we can move $u_2,u_2'$ from $U_2$, giving $|U_0\cup U_1|+2=\ell_1s=|V_1\setminus Y|$.

\textbf{Case 2.2.2.2.} $\exists \ell_1$, $\exists V_0'\subseteq V_0$ such that $|V_0'\cup V_1|=\ell_1s$. Choose $\ell_1$ to be minimal and note that since we are in Case 2.2.2. but not Case 2.2.2.1. we have $|V_1\setminus V_1^M|> \ell_1s-s$ and thus
\begin{equation}
\ell_1\geq k_1+1.\label{l1k1}
\end{equation}
Set $\ell_2:=m-\ell_1$. Since $|V_1\setminus V_1^M|>\ell_1s-s$, we reset $V_1:=V_1\setminus V_1^M$, $V_0:=V_0\cup V_1^M$ and set $b_1:=|V_1|-(\ell_1s-s)$. 

\textbf{Case 2.2.2.2.1.} $|U_0\cup U_2|<\ell_2 s$.  Set $a_1:=|U_1|-\ell_1s$.  Then we have $\delta(V_2, U_1)\geq k_2s+2s-5-r-(\ell_2s-a_1)=(k_2-\ell_2)s+2s-5-r+a_1\geq 2s-4+a_1\geq s-1+a_1$, and thus we are done by Lemma \ref{STARSlemma}(ii).  

\textbf{Case 2.2.2.2.2.} $|U_0\cup U_2|\geq \ell_2 s$.  If $|U_2|\leq \ell_2s$, then there exists $U_0'\in U_0$ such that $|U_1\cup U_0'|=\ell_1s=|V_1\cup Y|$.  Otherwise $|U_2|>\ell_2s$.  Set $a_2:=|U_2|-\ell_2s$.  Note that if $\ell_2\geq \ell_1$, then $\ell_2s\geq \frac{n}{2}$ and consequently $\delta(V_1, U_2)\geq \frac{n+3s-4}{2}-(\ell_1s-a_2)\geq \frac{3s-4}{2}+a_2\geq s-1+a_2$.  Then by Lemma \ref{STARSlemma}(ii) we can move $a_2$ vertices from $U_2$ and we are done.  So for the rest of this case we may suppose that
\begin{equation}
\ell_2\leq \ell_1-1.\label{l2l1}
\end{equation}

Since $|U_2|=\ell_2s+a_2$, we have 
\begin{equation}\label{22222:V1U2}
\delta(V_1, U_2)+\delta(U_2, V_1)\geq n+3s-5-(\ell_1s-a_2+\ell_2s+s-b_1)=2s-5+a_2+b_1.  
\end{equation}
If $\delta(V_1, U_2)\geq s$ or $\delta(U_2, V_1)\geq s$, then we can apply Lemma \ref{STARSlemma}(i) or (iii) to get a set of $a_2$ vertex disjoint $s$-stars from $U_2$ to $V_1$, giving $|U_1|+a_2=\ell_1s=|V_0'\cup V_1|$. So suppose for the rest of the case that 
\begin{equation}\label{s-1:V1U2}
\delta(V_1, U_2)\leq s-1 \text{ and } \delta(U_2, V_1)\leq s-1.
\end{equation} 
Thus \eqref{22222:V1U2} and \eqref{s-1:V1U2} imply $2\leq a_2+b_1\leq 3$.  Furthermore, if $\delta(V_1, U_2)+\delta(U_2, V_1)= 2s-2$, then we have $\delta(V_1, U_2)= s-1$ and $\delta(U_2, V_1)= s-1$.  

\begin{claim}\label{Claim 22222}
If $|U_1|\leq \ell_1s-s$, then there exists $U_0'\subseteq U_0$ such that $|U_1\cup U_0'|=\ell_1s-s$.  If $|U_1|\geq \ell_1s-s+1$, then there exists a set of vertex disjoint $s$-stars with centers $C\subseteq U_1$ and leaves in $V_2$ such that $|U_1\setminus C|=\ell_1s-s$ or else $\delta(V_1, U_2)\geq s-2+a_2$.

\end{claim}

\begin{proof}
First suppose $|U_1|\leq \ell_1s-s$.  Since $|U_2|=\ell_2s+a_2\leq \ell_2s+2<\ell_2s+s$, there exists $U_0'\subseteq U_0$ such that $|U_0'\cup U_1|=\ell_1s-s$.  Now suppose $|U_1|\geq \ell_1s-s+1$ and set $a_1:=|U_1|-(\ell_1s-s)$.  If $\ell_1\geq k_1+2$, then 
\begin{equation}\label{claim:V2U1}
\delta(V_2, U_1)\geq k_2s+2s-5-r-(\ell_2s+s-a_1)=(k_2-\ell_2)+s-5-r+a_1\geq 2s-4+a_1\geq s-1+a_1.  
\end{equation}
Thus we may apply Lemma \ref{STARSlemma}(ii) to get a set of $a_1$ vertex disjoint $s$-stars from $U_1$ to $V_2$ giving $|U_1|-a_1=\ell_1s-s$.  So suppose $\ell_1\leq k_1+1$, which implies $\ell_1=k_1+1$ by \eqref{l1k1}.  Consequently $\ell_2=k_2-1$.  By \eqref{l2l1}, we have $k_2-1=\ell_2\leq \ell_1-1=k_1$.  By \eqref{claim:V2U1}, we have $\delta(V_2, U_1)\geq 2s-5-r+a_1$.  If $k_2=k_1$, then $r\leq \frac{s-6}{2}$ and thus $\delta(V_2, U_1)\geq s-1+a_1$. So suppose $k_2=k_1+1$, which implies $r\leq s-3$ by Claim \ref{k2approxk1}.  If $r\leq s-4$, then \eqref{claim:V2U1} gives $\delta(V_2, U_1)\geq s-1+a_1$.  So suppose $r=s-3$.  If $a_1\geq 2$, we have $\delta(V_2, U_1)\geq s$.  Otherwise $a_1=1$ and $\delta(V_1, U_2)\geq k_2s+2s-5-r-(\ell_1s-a_2)\geq s-2+a_2$.
\end{proof}

$a_2=1$, $b_1=2$.  In this case, $|V_1|>|U_2|$ by \eqref{l2l1} and since $\delta(V_1, U_2)\geq s-1$, there is a vertex $u_2\in U_2$ such that $\deg(u, V_1)\geq s$ and we are done.

$a_2=2$, $b_1=1$.  If there is a vertex $v\in V_1$ with $\deg(v, U_2)\geq s$, then we apply Claim \ref{Claim 22222} to either finish or get $\delta(V_1, U_2)\geq s-2+a_2$.  However, if $\delta(V_1, U_2)\geq s-2+a_2$, then the fact that $a_2=2$, contradicts \eqref{s-1:V1U2}. So suppose $\Delta(V_1, U_2)\leq s-1$.  Since $\delta(U_2, V_1)\leq s-1$, there exists $u\in U_2$ such that for all $v\in V_1$ we have $$n+3s-5\leq \deg(v)+\deg(u)\leq \ell_1s+s-1+s-1+\ell_2s-2+s-1=n+3s-5,$$ thus $G[V_1, U_0\cup U_1]$ is complete.  Let $v_0,v_0'\in V_0$.  Let $u_2\in N(v_0)\cap U_2$ and choose a set of $s-1$ vertices $L\subseteq N(u_2)\cap V_1$.  Since $\Delta(V_1, U_2)\leq s-1$, there exists $u_2'\in N(v_0')\cap U_2$ such that $\deg(u_2', V_1\setminus L)\geq s-1$.  Let $L'$ be a set of $s-1$ vertices in $N(u_2')\cap (V_1\setminus L)$.  Since $G[V_1, U_0\cup U_1]$ is complete we can move $u_2$ and $u_2'$.

$a_2=1$, $b_1=1$.  If there is a vertex $v_1\in V_1$ with $\deg(v_1, U_2)\geq s$, then we apply Claim \ref{Claim 22222} to either finish or get $\delta(V_1, U_2)\geq s-2+a_2$. Since $a_2=1$, we have $\delta(V_1, U_2)\geq s-1$.  Since $|V_1|\geq |U_2|$, $\delta(V_1, U_2)\geq s-1$, and $\deg(v_1, U_2)\geq s$, there exists a vertex $u_2\in U_2$ such that $\deg(u_2, V_1)\geq s$ and we are done.  So we may suppose $\Delta(V_1, U_2),\Delta(U_2, V_1)\leq s-1$.  This implies that $\delta(U_2, V_0\cup V_2)\geq |V_0\cup V_2|-1$ and $\delta(V_1, U_0\cup U_1)\geq |U_0\cup U_1|-1$.  Since $\delta(V_1, U_2)+\delta(U_2, V_1)\geq 2s-3$, we can choose $u_2\in U_2$ such that $\deg(u_2, V_1)\geq s-1$.  Let $v_0\in V_0\cap N(u_2)$, which exists since $|V_0|\geq s-1$ and $\delta(U_2, V_0\cup V_2)\geq |V_0\cup V_2|-1$.  We have $\deg(v_0, U_1)>2s-2$ and thus $G[N(u_2)\cap V_1, N(v_0)\cap U_1]$ contains a copy of $K_{s-1,s-1}$.  This allows us to move one vertex from $U_2$ as needed.

\noindent
\textbf{Case 3}  For some $\ell_1\geq k_1$, we have $\ell_1s<|V_1\setminus V_1^M|\leq |V_0\cup V_1|<\ell_1s+s$.  Set $b_1:=|V_1\setminus V_1^M|-\ell_1s>0$ and $b_2:=|V_2|-(\ell_2s-s)$.  Reset $V_1:=V_1\setminus V_1^M$ and $V_0:=V_0\cup V_1^M$.  Set $\ell_2=m-\ell_1$.

\noindent
\textbf{Case 3.1} $|U_2\setminus U_2^M|\geq \ell_2s$.  Let $a_2:=|U_2\setminus U_2^M|-\ell_2s$.  Reset $U_2:=U_2\setminus U_2^M$ and $U_0:=U_0\cup U_2^M$.  We have 
\begin{equation}
\delta(V_1, U_2)+\delta(U_2, V_1)\geq 3s-5+a_2+b_1\geq 2s-2+a_2+b_1.
\end{equation}
Note that $a_2\geq 0$, $b_1>0$, so we are done by Lemma \ref{lemma:diagonalsum}.

\noindent
\textbf{Case 3.2} $|U_2\setminus U_2^M|<\ell_2s$.  Reset $U_2:=U_2\setminus U_2^M$ and $U_0:=U_0\cup U_2^M$.  We have $|U_1\cup U_0|>\ell_1s$.  

\textbf{Case 3.2.1.} $|U_0\cup U_1|\geq \ell_1s+s$.

\textbf{Case 3.2.1.1.} First suppose that $|U_1|\leq \ell_1s$.  Let $\bar{V}_i=\{v\in V_i:\deg(v, U_{3-i})\geq s\}$.  If $|\bar{V}_1|\geq \frac{n}{8}$ or $|\bar{V}_2|\geq \frac{n}{8}$, then we either get a set of $b_1$ vertex disjoint $s$-stars from $\bar{V}_1$ to $U_2$ or a set of $b_2$ vertex disjoint $s$-stars from $\bar{V}_2$ to $U_1$ by Lemma \ref{STARSlemma}(i).  Since $|U_1|\leq \ell_1s$ and $\ell_1s+s\leq |U_0\cup U_1|$ we can choose a set $U_0'\subseteq U_0$ such that $|(U_0\cup U_1)\setminus U_0'|=\ell_1s$ or we can choose a set $U_0'\subseteq U_0$ such that $|(U_0\cup U_1)\setminus U_0'|=\ell_1s+s$.  For $i=1,2$, let $\tilde{V}_i=\{v\in V_i\setminus \bar{V}_i:\deg(v, U_1\cup U_2)\leq |U_i|+s-2\}$.  We have 
\begin{align}
\delta(\tilde{V}_1, U_0)+\delta(\tilde{V}_2, U_0)\geq n+3s-4-(|U_1|+s-2+|U_2|+s-2)=|U_0|+s\label{gadget1}
\end{align}

 
If $|\tilde{V}_1|\geq \frac{n}{8}$ and $|\tilde{V}_2|\geq \frac{n}{8}$, then by (\ref{gadget1}) and Lemma \ref{K_1sum} we can find a $K_{s,s}$ with $b_1$ vertices in $V_1$ and $s-b_1$ vertices in $V_2$.  Then we choose $U_0'\subseteq U_0$ such that $|V_1|-b_1=\ell_1s=|(U_0\cup U_1)\setminus U_0'|$.  Otherwise we have $|\tilde{V}_1|<\frac{n}{8}$ or $|\tilde{V}_2|<\frac{n}{8}$.  Suppose that $|\tilde{V}_1|<\frac{n}{8}$.  First note that for all $v\in V_1\setminus (\bar{V}_1\cup \tilde{V}_1)$, $\deg(v, U_2)=s-1$.  Since $|V_1\setminus (\bar{V}_1\cup \tilde{V}_1)|>\frac{n}{8}$, we can apply Lemma \ref{STARSlemma}(i) to get a set of $b_1$ vertex disjoint $(s-1)$-stars from $V_1\setminus (\bar{V}_1\cup \tilde{V}_1)$ to $U_2$.  Let $v_1, v_2,\dots, v_{b_1}$ be the centers and $L(v_i)$ be the leaf sets for each star.  

If $|\tilde{V}_2|\geq \frac{n}{8}$, then for every star we have $|N(L(v_i))\cap \tilde{V}_2|>\frac{n}{16}$ and for all $\tilde{v}\in N(L(v_i))\cap \tilde{V}_2$ we have $$n+3s-4\leq \deg(v_i)+\deg(\tilde{v})\leq |U_1|+s-1+\deg(v_i, U_0)+|U_2|+s-2+\deg(\tilde{v}, U_0),$$ which implies $\deg(v_i, U_0)+\deg(\tilde{v}, U_0)\geq |U_0|+s-1$.  So for each $v_i$, we can find a $K_{s-1,s-1}$ with $s-1$ vertices in $N(v_i)\cap U_0$ and $s-1$ vertices in $N(L(v_i))\cap \tilde{V}_2$.  Since we only need to move at most $s-1$ vertices from $V_1$, we can always choose a unique vertex from $U_0$ for each center in $V_1$ to complete the copy of $K_{s,s}$.  

If $|\tilde{V}_2|< \frac{n}{8}$, then  $|V_i\setminus (\bar{V}_i\cup \tilde{V}_i)|>\frac{n}{8}$ for $i=1,2$.  Set $V_i':=V_i\setminus (\bar{V}_i\cup \tilde{V}_i)$ for $i=1,2$.  We know that $\min\{b_1, s-b_1\}\leq \frac{s}{2}$ and since $s\geq 3$,  $\min\{b_1, s-b_1\}\leq s-2$.  Without loss of generality, suppose $b_1\leq s-b_1$.  Since $|V_1'|>\frac{n}{8}$, we start by taking a set of $b_1$ vertex disjoint $(s-1)$-stars from $V_1'$ to $U_2$.  Let $v_1, v_2,\dots, v_{b_1}$ be the centers and $L(v_i)$ be the leaf sets for each star.  For every star we have $|N(L(v_i))\cap V_2'|>\frac{n}{16}$ and for all $v'\in N(L(v_i))\cap V_2'$ we have $$n+3s-4\leq \deg(v_i)+\deg(v')\leq |U_1|+s-1+\deg(v_i, U_0)+|U_2|+s-1+\deg(v', U_0),$$ which implies $\deg(v_i, U_0)+\deg(v', U_0)\geq |U_0|+s-2$. So for each $v_i$, we can find a $K_{s-2,s-1}$ with $s-2$ vertices in $U_0\cap N(v_i)$ and $s-1$ vertices in $N(L(v_i))\cap V_2'$.  Since we only need to move at most $s-2$ vertices from $V_1$, we can always choose a unique vertex from $U_0$ for each center in $V_1$ to complete the copy of $K_{s,s}$.

\textbf{Case 3.2.1.2.} $|U_1|>\ell_1s$.  Let $a_1:=|U_1|-\ell_1s$.  In this case we have 
\begin{equation}\label{3212}
\delta(V_2, U_1)\geq k_2s+2s-5-r-(\ell_2s-a_1)=(k_2-\ell_2)s+2s-5-r+a_1.
\end{equation}

\textbf{Case 3.2.1.2.1.} $\ell_1>k_1$. Then $\ell_2<k_2$ and \eqref{3212} gives $\delta(V_2, U_1)\geq s-1+a_1$ and we are done by moving vertices to $V_1$. 

\textbf{Case 3.2.1.2.2.} $\ell_1=k_1$ and so $\ell_2=k_2$.  

Suppose $k_2=k_1$. Then $r\leq \frac{s-6}{2}$ and we have $\delta(V_2, U_1)\geq s-1+a_1$ so we are done by moving vertices to $V_1$.

Suppose $k_2=k_1+1$. This implies $r\leq s-3$. Now we have 
$\delta(V_2, U_1)\geq s-2+a_1$.  
If $\delta(V_2, U_1)\geq s$, then we would be done by moving vertices to $V_1$. So suppose $a_1=1$ and $r=s-3$. Recall $b_2=|V_2|-(k_2s-s)$.  We have 
$\delta(U_1, V_2)\geq k_1s+s+r-(k_1s+s-b_2)=s-3+b_2$, 
so we would be done by moving vertices to $V_1$ unless $1\leq b_2\leq 2$.  Furthermore, we have 
\begin{equation}\label{s-3+b1}
\delta(U_2, V_1)\geq k_1s+s+r-(k_1s+s-b_1)=s-3+b_1
\end{equation}

Suppose $b_2=2$.  Since $a_1=1$ and $k_2=k_1+1$ we have $|V_2|>|U_1|$.  Since $\delta(V_2, U_1)\geq s-1$, there exists a vertex $u_1\in U_1$ such that $\deg(u_1, V_2)\geq s$.  If $b_1\geq 3$, then \eqref{s-3+b1} implies $\delta(U_2, V_1)\geq s$ and thus we can move $b_1$ vertices from $V_1$ by Lemma \ref{STARSlemma}(iii).  Otherwise let $V_2'=\{v\in V_2:\deg(v, U_1)\leq s-1\}$.  If $|V_2\setminus V_2'|>2s\alpha^{1/3}k_2s$, then since $\Delta(U_1, V_2)\leq 2\alpha^{1/3}k_2s$ there would be two vertex disjoint $s$-stars from $V_2\setminus V_2'$ to $U_1$.  So suppose $|V_2'|> \frac{n}{4}$.  Note that for all $v\in V_2'$, $\deg(v, U_0\cup U_2)\geq k_2s+2s-5-r-(s-1)=k_2s-1=|U_0\cup U_2|$, so $G[V_2', U_0\cup U_2]$ is complete. If $b_1=1$, then since $\delta(V_1, U_0\cup U_2)\geq 2s-3\geq s$ we can move a vertex from $V_1$, giving $|U_1\setminus\{u_1\}|=k_1s=|V_1|-1$.  So suppose $b_1=2$.  If there is a vertex $v_1\in V_1$ such that $\deg(v_1, U_0\cup U_2)\geq 2s$, then we would be done since $\delta(V_1, U_0\cup U_2)\geq 2s-3\geq s$ and $G[V_2', U_0\cup U_2]$ is complete so we can move two vertices from $V_1$.  So suppose $\Delta(V_1, U_0\cup U_2)\leq 2s-1$.  Then $\delta(V_1, U_1)\geq k_2s+2s-5-r-(2s-1)=k_2s-s-1=k_1s-1=|U_1|-2$.  Since $b_1=2$, we have $\delta(U_2, V_1)\geq s-1$ by \eqref{s-3+b1}.  Thus there are two vertex disjoint $s$-stars from $U_2$ to $V_1$ with leaf sets $L_1$ and $L_2$.  Let $\tilde{U}_1:=U_1\cap (N(L_1)\cap N(L_2))$ and note that since $\delta(V_1, U_1)\geq |U_1|-2$, we have $|\tilde{U}_1|\geq |U_1|-4s$.  Now since $\delta(V_2', U_1)\geq s-1$ and $\Delta(U_1, V_2)\leq 2\alpha^{1/3}k_2s$, there exist two vertex disjoint $(s-1)$-stars from $V_2'$ to $\tilde{U}_1$.  Since $G[\tilde{U}_1, L_1\cup L_2]$ and $G[V_2', U_0\cup U_2]$ are complete, we can move two vertices from $V_2$ to $V_1$ and $U_2$ to $U_1$.  We finish by moving $s-3$ vertices from $U_0$ to $U_1$ and $s-4$ vertices from $V_0$ to $V_1$, giving $|U_1|+2+s-3=k_1s+s=|V_1|+2+s-4$.

%

Suppose $b_2=1$.  If there exists a vertex $v_2\in V_2$ such that $\deg(v_2, U_1)\geq s$, then we would be done by moving $v_2$ to $V_1$.  So suppose $\Delta(V_2, U_2)\leq s-1$ and thus $\delta(V_2, U_0\cup U_2)\geq k_2s+2s-5-r-(s-1)=k_2s-1=|U_0\cup U_2|$.  Let $v_2\in V_2$ and let $L$ be the set of leaves in $U_1$ of an $(s-1)$-star with center $v_2$.  Let $V_1'=N(L)\cap V_1$ and note that $|V_1'|\geq |V_1|-2s\alpha^{1/3}k_1s$.  Since $\delta(V_1', U_0\cup U_2)\geq k_2s+2s-5-r-(k_1s+1)=2s-3\geq s$, there exists a vertex $u_2\in U_0\cup U_2$ such that $\deg(u, V_1')\geq s-1$.  Since $G[V_2, U_0\cup U_2]$ is complete, we can move $v_2$ and $u_2$.  We finish by moving $s-2$ vertices from $U_0$ to $U_1$ and $s-1-b_1$ vertices from $V_0$ to $V_1$ giving $|U_1|+1+s-2=k_1s+s=|V_1|+1+s-1-b_1$.  


Finally, suppose $k_2\geq k_1+2$.  Here we have $\delta(U_1, V_2)\geq k_1s+s+r-(k_1s+s-b_2)=r+b_2$.  If $r\geq s-b_2$, then $\delta(U_1, V_2)\geq s$ and we would be done by moving vertices from $V_2$ to $V_1$, so suppose $r\leq s-1-b_2$.  Then we have 
\begin{equation}\label{321223}
\delta(V_2, U_1)\geq k_2s+2s-5-r-(k_2s-a_1)\geq s-4+a_1+b_2.  
\end{equation}
We would have $\delta(V_2, U_1)\geq s$ and be done unless $2\leq a_1+b_2\leq 3$.

Suppose $a_1=2$, $b_2=1$.  If $r\leq s-3$, then $\delta(V_2, U_1)\geq s$ by \eqref{321223}, so suppose $r= s-2$.  We have $\delta(U_1, V_2), \delta(V_2, U_1)\geq s-1$ and $\delta(V_1, U_0\cup U_2)\geq k_2s+2s-5-r-(k_1s+2)\geq 3s-5$.  If there was a vertex $v_2\in V_2$ such that $\deg(v_2, U_1)\geq s$, then we would be done by moving $v_2$ to $V_1$.  So suppose $\Delta(V_2, U_1)\leq s-1$ and thus $\delta(V_2, U_0\cup U_2)\geq k_2s+2s-5-r-(s-1)=k_2s-2=|U_0\cup U_2|$. Let $v_2\in V_2$ and let $L:=N(v_2)\cap U_1$.  Every vertex in $N(L)\cap V_1=:V_1'$ has at least $3s-5\geq s$ neighbors in $U_0\cup U_2$, so there exists a vertex $u_2\in U_0\cup U_2$ such that $\deg(u_2, V_1')\geq 3s-5\geq s-1$.  Then since $G[V_2, U_0\cup U_2]$ is complete, we have a copy of $K_{s,s}$ which allows us to move $v_2$.  We finish by moving $s-3$ vertices from $U_0$ to $U_1$ and $s-1-b_1$ vertices from $V_0$ to $V_1$ giving $|U_1|+1+s-3=k_1s+s=|V_1|+1+s-1-b_1$.

Suppose $a_1=1$, $b_2=2$.  If $r\leq s-4$, then $\delta(V_2, U_1)\geq s$ by \eqref{321223}, so suppose $r= s-3$.  We have $\delta(U_1, V_2), \delta(V_2, U_1)\geq s-1$ and $\delta(V_1, U_0\cup U_2)\geq k_2s+2s-5-r-(k_1s+1)\geq 3s-3$.  Let $V_2'=\{v\in V_2:\deg(v, U_1)\leq s-1\}$.  If $|V_2\setminus V_2'|>2s\alpha^{1/3}k_2s$, then since $\Delta(U_1, V_2)\leq 2\alpha^{1/3}k_2s$ there would be two vertex disjoint $s$-stars from $V_2\setminus V_2'$ to $U_1$, so suppose not.  Then $|V_2'|>\frac{n}{4}$.  Note that $G[V_2', U_0\cup U_2]$ is complete.  Since $|V_2|>|U_1|$ and $\delta(V_2, U_1)\geq s-1$, there exists a vertex $u_1\in U_1$ such that $\deg(u_1, V_2)\geq s$.  Now we must move $b_1$ vertices from $V_1$.  If say $\frac{n}{8}$ vertices in $V_1$ have at least $s$ neighbors in $U_0$, then we can find a $K_{s,s}$ with $s$ vertices in $U_0$, $b_1$ vertices in $V_1$ and $s-b_1$ vertices in $V_2$ by Lemma \ref{gadget1} and the fact that $G[V_2', U_0\cup U_2]$ is complete.  Otherwise we have $\frac{n}{4}$ vertices with at most $s-1$ neighbors in $U_0$ and consequently at least $3s-3-(s-1)\geq s$ neighbors in $U_2$.  Either way there exists $b_1$ vertex disjoint $s$-stars from $V_1$ to $U_2$.

Suppose $a_1=1=b_2$.  If there is a vertex in $V_2$ with $s$ neighbors in $U_1$, then we would be done, so suppose not.  Since $b_2=1$, we have $r\leq s-2$.  If $r=s-2$, then $\delta(U_1, V_2)\geq s-1$.  If $r\leq s-3$, then $\delta(V_2, U_1)\geq s-1$.  So either way there is a vertex $v_2\in V_2$ such that $\deg(v_2, U_1)= s-1$.  Let $L:=N(v_2)\cap U_1$.  We have $\delta(V_2, U_0\cup U_2)\geq k_2s+2s-5-r-(s-1)\geq k_2s-2=|U_0\cup U_1|-1$.  Since $\delta(V_1, U_0\cup U_2)\geq 3s-4$, every vertex in $N(L)\cap V_1=:V_1'$ has at least $3s-5$ neighbors in $N(v_2)\cap (U_0\cup U_2)$. So there exists a vertex $u_2\in N(v_2)\cap (U_0\cup U_2)$ with at least $3s-5\geq s-1$ neighbors in $V_1'$.  This gives us a copy of $K_{s,s}$ which allows us to move $v_2$.

\textbf{Case 3.2.2.} $\ell_1s<|U_0\cup U_1|<\ell_1s+s$.  

\textbf{Case 3.2.2.1.} $|U_1|\leq \ell_1s$.  Thus there exists $U_0'\subseteq U_0$ such that $|(U_0\cup U_1)\setminus U_0'|=\ell_1s$.  So we try to make $|V_1|=\ell_1s$ or $|V_2|=\ell_2s$.  Recall $\ell_2=m-\ell_1$ and $b_1=|V_1|-\ell_1s$.  Let $a_2:=|U_2|-(\ell_2s-s)$.  We have 
\begin{equation}
\delta(V_1, U_2)+\delta(U_2, V_1)\geq n+3s-5-(\ell_1s+s-a_2+\ell_2s-b_1)=2s-5+a_2+b_1.  
\end{equation}
If $\delta(V_1, U_2)\geq s$ or $\delta(U_2, V_1)\geq s$, then we would be able to find $b_1$ vertex disjoint $s$-stars from $V_1$ to $U_2$ by Lemma \ref{STARSlemma}(i) or (iii) and we are done.  So suppose $\delta(V_1, U_2)\leq s-1$ and $\delta(U_2, V_1)\leq s-1$, thus $2\leq a_2+b_1\leq 3$.  If $\delta(V_1, U_2)+\delta(U_2, V_1)= 2s-2$, then we have $\delta(V_1, U_2)= s-1$ and $\delta(U_2, V_1)= s-1$.  
Furthermore, we have 
\begin{equation}\label{U0U1<->V0V2-3221}
\delta(U_0\cup U_1, V_0\cup V_2)+\delta(V_0\cup V_2, U_0\cup U_1)\geq n+3s-5-(\ell_1s+b_1+\ell_2s-s+a_2)=4s-5-a_2-b_1.
\end{equation}
Let $U_2':=\{u\in U_2: \deg(u, V_1)\leq s-1\}$.

Suppose $a_2=2$, $b_1=1$.  If there is a vertex $v_1\in V_1$ with $\deg(v_1, U_2)\geq s$, then we are done by moving $v_1$ to $V_2$.  If $e(U_2, V_1)>(s-1)|V_1|$, then there exists a vertex $v_1\in V_1$ such that $\deg(v_1, U_2)\geq s$, so suppose not. If $|U_2\setminus U_2'|>3\alpha^{2/3}k_2s$, then since $|V_1|-|U_2|\leq 2\alpha^{2/3}k_2s$ we have $e(U_2, V_1)>(s-1)|V_1|$, so suppose not.  Then $|U_2'|\geq |U_2|-3\alpha^{2/3}k_2s$.  For all $v\in V_1$ and $u\in U_2'$ we have 
\begin{equation}\label{32211}
n+3s-5\leq \deg(v)+\deg(u)\leq \ell_1s+s-1+s-1+\ell_2s-2+s-1=n+3s-5,
\end{equation}
thus $G[V_1, U_0\cup U_1]$ is complete and $G[U_2', V_0\cup V_2]$ is complete.  Since $\delta(U_2', V_1)\geq s-1$, there exists a vertex $v_1\in V_1$, such that $\deg(v_1, U_2')= s-1$.  Let $u_0\in U_0$ and note that $\deg(u_0, V_2)>s$.  Since $G[V_1, U_0\cup U_1]$ is complete we can move $v_1$ from $V_1$ along with $u_0$.

Suppose $a_2=1, b_1=2$.  First suppose that there exists $v_1\in V_1$ with at least $s$ neighbors in $U_2$.  Let $L\subseteq N(v_1)\cap U_2$ with $|L|=s$.  In this case we can apply the argument of the previous paragraph to the sets $V_1\setminus v_1$ and $U_2\setminus L$.  
So suppose that $\Delta(V_1, U_2)\leq s-1$ and $|U_2'|\geq |U_2|-2\alpha^{2/3}k_2s$.  Equation \eqref{32211} holds which implies that $G[V_1, U_0\cup U_1]$ is complete and $G[U_2', V_0\cup V_2]$ is complete.  Every vertex in $U_2'$ has $s-1$ neighbors in $V_1$, so there are two vertex disjoint $(s-1)$-stars from $V_1$ to $U_2'$ with centers $v_1$ and $v_1'$.  Since $G[V_1, U_0\cup U_1]$ is complete and $|U_0|\geq s-1\geq 2$, there exist $u_0, u_0'\in U_0$.  Since $\deg(u_0, V_2), \deg(u_0', V_2)>2s$, we can move $v_1$ and $v_1'$ by taking $u_0$ and $u_0'$.  Then let $U_0'\subseteq U_0$ so that $|U_1|+|U_0'|=\ell_1s=|V_1|-2$.


Suppose $a_2=1, b_1=1$.  If there is a vertex $v_1\in V_1$ such that $\deg(v_1, U_2)\geq s$, then we can move $v_1$ to $V_2$ and be done, so suppose $\Delta(V_1, U_2)\leq s-1$.  First suppose that $\Delta(U_2, V_1)\leq s-1$.  For all $v\in V_1$ and $u\in U_2$ we have $n+3s-5\leq \deg(u)+\deg(v)\leq \ell_1s+s-1+s-1+\ell_2s-1+s-1=n+3s-4$.  Thus $\delta(V_1, U_0\cup U_1)\geq |U_0\cup U_1|-1$ and $\delta(U_2, V_0\cup V_2)\geq |V_0\cup V_2|-1$.  Let $v_1\in V_1$ such that $\deg(v_1, U_2)= s-1$, which exists since $\delta(V_1, U_2)\geq s-1$ or $\delta(U_2, V_1)\geq s-1$.  Let $L:=N(v_1)\cap U_2$ and $V_2':=N(L)\cap V_2$; note that $|V_2'|\geq |V_2|-s$ since $\delta(U_2, V_0\cup V_2)\geq |V_0\cup V_2|-1$.  Finally let $u_0\in U_0\cap N(v_1)$, which exists since $\delta(V_1, U_0\cup U_1)\geq |U_0\cup U_1|-1$ and $|U_0|\geq s-1$.  Since $\deg(u_0, V_2')>s$, we can move $v_1$ along with $u_0$.  So we may suppose that there exists some $u_2\in U_2$ such that $\deg(u_2, V_1)\geq s$.  Let $V_2':=\{v\in V_2:\deg(v, U_1)\leq s-1\}$.  If say $|V_2\setminus V_2'|>\frac{n}{8}$, then since $\Delta(U_1, V_2)\leq 2\alpha^{1/3}k_2s$ we could move $b_2$ vertices from $V_2$ and we would be done. So we may suppose that $|V_2'|>\frac{n}{4}$.  Note that we have
\begin{equation}\label{32213}
\delta(V_1, U_0)+\deg(V_2', U_0)\geq n+3s-4-(|U_1|+s-1+|U_2|+s-1)=|U_0|+s-2.
\end{equation}
Let $v_1\in V_1$ such that $\deg(v_1, U_2)=s-1$ and let $L:=N(v_1)\cap U_2$.  Let $\tilde{V}_2:=V_2'\cap N(L)$ and note that $|\tilde{V}_2|>\frac{n}{8}$.  For all $\tilde{v}\in \tilde{V}_2$ we have $\deg(\tilde{v}, N(v_1)\cap U_0)\geq s-2$ by \eqref{32213}.  Since $|\tilde{V}_2|>|N(v_1)\cap U_0|$, there exists $u_0\in N(v_1)\cap U_0$ such that $\deg(u_0, \tilde{V}_2)\geq s-1$.  This completes a copy of $K_{s,s}$ which allows us to move $v_1$.

\textbf{Case 3.2.2.2.} $|U_1|>\ell_1s$.  Let $a_1:=|U_1|-\ell_1s$.  Recall $\ell_2=m-\ell_1$, $b_1=|V_1|-\ell_1s$, $a_2=|U_2|-(\ell_2s-s)$, and $b_2:=|V_2|-(\ell_2s-s)$.  We have
\begin{equation}\label{3222:V1U2}
\delta(V_1, U_2)+\delta(U_2, V_1)\geq n+3s-5-(\ell_1s+s-a_2)-(\ell_2s-b_1)=2s-5+a_2+b_1
\end{equation}
and
\begin{equation}\label{3222:V2U1}
\delta(V_2, U_1)+\delta(U_1, V_2)\geq n+3s-5-(\ell_2s-a_1)-(\ell_1s+s-b_2)=2s-5+a_1+b_2
\end{equation}

\textbf{Case 3.2.2.2.1.} For some $i\in\{1,2\}$ we have $\delta(V_i, U_{3-1})\geq s$ or $\delta(U_{3-i}, V_i)\geq s$.  Without loss of generality (all cases are similar, but not exactly the same), suppose $\delta(V_2, U_1)\geq s$.  This implies by Lemma \ref{STARSlemma}(iii) that there is a set of $a_1$ vertex disjoint $s$-stars from $U_1$ to $V_2$ and a set of $b_2$ vertex disjoint $s$-stars from $V_2$ to $U_1$.  So if we can move $a_2$ vertices from $U_2$ or $b_1$ vertices from $V_1$, then we say that we are done.  If $\delta(V_1, U_2)\geq s$ or $\delta(U_2, V_1)\geq s$, then we can apply Lemma \ref{STARSlemma}(i) or (iii) and we are done, so suppose not.  This implies $2\leq a_2+b_1\leq 3$ by \eqref{3222:V1U2}.  Furthermore, if $a_2+b_1=3$, then $\delta(V_1, U_2)+\delta(U_2, V_1)\geq 2s-2$ and we may suppose $\delta(V_1, U_2)=s-1$ and $\delta(U_2, V_1)=s-1$.  Let $U_2':=\{u\in U_2:\deg(u, V_1)\leq s-1\}$ and $V_1':=\{v\in V_1:\deg(v, U_2)\leq s-1\}$.  

Since $2\leq a_2+b_1\leq 3$, either $a_2=1$ or $b_1=1$.  Without loss of generality suppose $a_2=1$ and thus $1\leq b_1\leq 2$.  If there is a vertex $u_2\in U_2$ such that $\deg(u_2, V_1)\geq s$, then we can move $u_2$ and we are done, so suppose $\Delta(U_2, V_1)\leq s-1$.  For all $u\in U_2$ and $v\in V_1'$ we have $n+3s-5\leq \deg(u)+\deg(v)\leq \ell_1s+s-1+s-1+\ell_2s-b_1+s-1\leq n+3s-4$ and thus $\delta(U_2, V_0\cup V_2)\geq |V_0\cup V_2|-1$ and $\delta(V_1', U_0\cup U_1)\geq |U_0\cup U_1|-1$.  If $b_1=1$, then we may suppose $\Delta(V_1, U_2)\leq s-1$ or else we are done.  In this case $V_1'=V_1$.  If $b_1=2$, then $\delta(V_1, U_2)\geq s-1$. If there are two vertex disjoint $s$-stars from $V_1$ to $U_2$, then we are done since $b_1\leq 2$.  This implies that $|V_1'|\geq |V_1|-2s\alpha^{1/3}k_2s$.  So in either case there exists a vertex $u_2\in U_2$ such that $\deg(u_2, V_1')=s-1$.  Since $\delta(V_2, U_1)\geq s$, there is a set of $s$ vertex disjoint $s$-stars from $N(u_2)\cap V_2$ to $U_1$.  Finally since $\delta(V_2', U_0\cup U_1)\geq |U_0\cup U_1|-1$, the leaf set of one of the $s$-stars from $V_2$ to $U_1$ will form a $K_{s-1,s-1}$ with $s-1$ vertices in $N(u_2)\cap V_1'$ and $s-1$ vertices in $U_1$.  Then we move $b_2-1$ more vertices from $V_2$.

\textbf{Case 3.2.2.2.2.} For all $i\in\{1,2\}$ we have $\delta(V_i, U_{3-i})\leq s-1$ and $\delta(U_{3-i}, V_i)\leq s-1$.  So by \eqref{3222:V1U2} and \eqref{3222:V2U1}, we may suppose $2\leq a_1+b_2\leq 3$ and $2\leq a_2+b_1\leq 3$.  We have 
\begin{equation}\label{32222a}
\delta(V_2, U_1)\geq k_2s+2s-5-r-(\ell_2s-a_1)=(k_2-\ell_2)s+2s-5-r+a_1\geq (k_2-\ell_2)s+s-4+a_1.
\end{equation}
If $\ell_1>k_1$, then $k_2>\ell_2$ and $\delta(V_2, U_1)\geq s$ by \eqref{32222a}.  So suppose $\ell_1=k_1$ and thus $\ell_2=k_2$.  We also have
\begin{equation}\label{32222b}
\delta(V_1, U_2)\geq k_2s+2s-5-r-(k_1s+s-a_2)=(k_2-k_1)s+s-5-r+a_2.
\end{equation}
If $k_2\geq k_1+2$, then $\delta(V_1, U_2)\geq s$ by \eqref{32222b}. So suppose $k_2\leq k_1+1$.  If $k_2=k_1$, then $r\leq \frac{s-6}{2}$ by Claim \ref{k2approxk1} and thus \eqref{32222a} gives $\delta(V_2, U_1)\geq 2s-5-\frac{s-6}{2}+a_1\geq s$. So suppose $k_2=k_1+1$ which implies $r\leq s-3$ by Claim \ref{k2approxk1}.  If $r\leq s-4$, then \eqref{32222a} implies $\delta(V_2, U_1)\geq s-1+a_1\geq s$.  So suppose $r=s-3$.  Finally if either $a_1\geq 2$ or $a_2\geq 2$, then \eqref{32222a} or \eqref{32222b} implies $\delta(V_1, U_2)\geq s$ or $\delta(V_2, U_1)\geq s$.  So suppose $a_1=1=a_2$ and thus $\delta(V_1, U_2)=s-1=\delta(V_2, U_1)$.  For $i=1,2$, let $V_i':=\{v\in V_i:\deg(v, U_{3-i})\leq s-1\}$.  For all $v\in V_i$, $\deg(v, U_0\cup U_i)\geq k_2s+2s-5-r-(s-1)=k_2s-1=|U_0\cup U_i|$, thus $G[V_i', U_0\cup U_i]$ is complete.

First suppose $b_1=2=b_2$. Since $|V_1|>|U_2|$ and $|V_2|>|U_1|$, there are vertices $u_1\in U_1$ and $u_2\in U_2$ such that $\deg(u_1, V_2)\geq s$ and $\deg(u_2, V_1)\geq s$.  If $|V_i\setminus V_i'|>2s\alpha^{1/3}k_2s$ for some $i$, then we would be done by moving two vertices from $V_i\setminus V_i'$ and moving $u_i$ from $U_i$ for some $i=1,2$.  So we may assume that $|V_i'|\geq |V_i|-s\alpha^{1/3}n$ for $=1,2$.  Since $\delta(V_1', U_2)\geq s-1$ and $|V_1'|\geq |V_1|-s\alpha^{1/3}n$, there exists $u_2\in U_2$ such that $\deg(u_2, V_1')\geq s-2$ and there exists $u_1\in U_1$ such that $\deg(u_1, V_2')\geq 2$.  Now since $G[V_1', U_0\cup U_1]$ and $G[V_2', U_0\cup U_2]$ are complete, we have a copy of $K_{s,s}$ with $s-2$ vertices in $V_1'$, $2$ vertices in $V_2'$, $s-2$ vertices in $U_0$, $1$ vertex in $U_1$ and $1$ vertex in $U_2$.  Then we move the remaining $s-4$ vertices from $V_0$ to $V_1$

Now suppose $b_i=2$ and $b_{3-i}=1$ for some $i$.  Without loss of generality, suppose $b_1=1$ and $b_2=2$.  Since $|V_2|>|U_1|$, there is a vertex $u_1\in U_1$ such that $\deg(u_1, V_2)\geq s$.  So we would be done unless $\Delta(V_1, U_2)\leq s-1$ and thus $V_1'=V_1$.  Let $u_2, u_2'\in U_2$ be the centers of two vertex disjoint $(s-1)$-stars from $U_2$ to $V_1$.  Then since $\delta(V_2, U_1)\geq s-1$ we can choose two vertex disjoint $(s-1)$-stars from $(N(u_2)\cap N(u_2'))\cap V_2$ to $U_1$.  Then since $G[V_1, U_0\cup U_1]$ is complete we are done.

Finally suppose $b_1=1=b_2$.  If there exists $v_2\in V_2$ (without loss of generality) such that $\deg(v, U_1)\geq s$, then there is a vertex $u_1\in U_1$ such that $\deg(u_1, V_2)\geq s$.  So we would be done unless $\Delta(V_1, U_2)\leq s-1$ and $\Delta(U_2, V_1)\leq s-1$.  Thus $G[V_1, U_0\cup U_1]$ is complete.  Let $u_2, u_2'\in U_2$ be the centers of two vertex disjoint $(s-1)$-stars from $U_2$ to $V_1$.  Then since $\delta(V_2, U_1)\geq s-1$ we can choose two vertex disjoint $(s-1)$-stars from $N(u_2)\cap N(u_2')\cap V_2$ to $U_1$.  Then since $G[V_1, U_0\cup U_1]$ is complete we are done.  Otherwise $\Delta(V_i, U_{3-i})\leq s-1$ for $i=1,2$ in which case $G[V_i, U_0\cup U_i]$ is complete for $i=1,2$.  Let $u_1\in U_1$ such that $\deg(u_1, V_2)\geq s-1$ and let $v_1u_2\in E(V_1, U_2)$.  Since $G[V_1, U_0\cup U_1]$ and $G[V_2, U_0\cup U_2]$ are complete, we have a copy of $K_{s,s}$ with $s-1$ vertices in $V_2$, $1$ vertex in $V_1$, $s-2$ vertices in $U_0$, $1$ vertex in $U_1$, and $1$ vertex in $U_2$.  Then we move the remaining $s-2$ vertices from $V_0$ to $V_2$.

\section{Examples when $\delta_U$ is constant}


Here we prove Proposition \ref{probexample}.  We ignore floors and ceilings since they are not vital to our calculations.

\begin{proof}
Given a positive integer $s$, let $c:=s^{1/3}$, $d:=2c$, $a := s^c$, and $b := \frac{s}{d}a=\frac{s^{c+1}}{d}$. Let $s$ be large enough so that $s^{2s^{2/3}}\left(\frac{(3d)^{d}}{s^{(c-1)s}}\right)^{s}<\frac{1}{2}$.
Let $A, B$ be sets such that $|A| = a$ and $|B| = b$. Consider the random bipartite graph by adding the pair from $A \times B$ with
probability $p := \frac{3d}{s}$ (all choices made independently). Then for $u\in A$,
$\mathbb{E}(\deg(u)) = pb = 3s^c$ and for $v\in B$, $\mathbb{E}(\deg(v)) = pa = 3ds^{c-1}$.  The probability that there exists $u\in A$ with $\deg(u) < 2s^c$ or $v \in B$ with $\deg(v) < 2ds^{c-1}$ 
is less than $1/2$ by a standard application of Chernoff's bound. In addition,
the probability that there exists $K_{d,s}$ with $d$ vertices in $A$ is at most
$$\binom{a}{d}\binom{b}{s}p^{ds}<a^d b^s p^{ds}=s^{cd} \frac{s^{(c+1)s}}{d^s} \frac{(3d)^{ds}}{s^{ds}}=\frac{s^{cd}}{s^{(d-(c+1))s}}\left(\frac{(3d)^d}{d}\right)^s \leq s^{2s^{2/3}}\left(\frac{(3d)^d}{s^{(c-1)s}}\right)^s <\frac{1}{2}. $$ 
Consequently there exists a graph $H$ on $A\cup B$ such that
\begin{itemize}
\item $\deg(u) \geq 2s^c$ for every $u \in A$, $\deg(v) \geq 2ds^{c-1}$ for $v \in B$ and
\item  $H$ has no $K_{d,s}$ with $d$ vertices in $A$.
\end{itemize}

Let $G$ be obtained from $H$ by adding a set $A'$ of $n - a$ vertices to $A$ and a set $B'$ of $n-b$ vertices to $B$ with $n$ large as usual. We add all edges between $A'$ and $B\cup B'$.  The sum of degrees in $G$ is at least $2s^c+(n-s^c)= n + s^c$.   

Suppose that $G$ can be tiled with $K_{s,s}$.  Since $G[A, B']$ is empty, any copy of $K_{s,s}$ touching $A$ must have $s$ vertices in $B$.  Also, any copy touching $A$ must have at most $d-1$ vertices from $A$, since $H$ has no $K_{d,s}$ with $d$ vertices in $A$.  So the number of copies touching $A$ is at least $\frac{a}{d-1}$.  However, this implies that $s\frac{a}{d-1}\leq |B|= \frac{s}{d}a$, a contradiction.
\end{proof}
In addition, it is possible to construct graphs $G$ for some small values of $s$ in which $\delta_U+\delta_V>n+2s-2\croot{s}+c(s)$ such that $G$ cannot be tiled with $K_{s,s}$ (see \cite{LouisThesis} for details).

\section{Conclusion}

In Theorem \ref{main 1} and Theorem \ref{main 2} we show that if $\delta(G)$ is $\Omega(n)$, then $\delta_U+\delta_V\geq n+3s-5$ suffices to tile $G$ with $K_{s,s}$.  The only example we have which shows $n+3s-5$ is best possible has the property that $\delta_U=\delta_V$.  When $\delta_V>\delta_U$ we have examples which show that we can't do better than $n+3s-7$.  This leaves open the question of whether $n+3s-6$ suffices when $\delta_V>\delta_U$.

In Theorem \ref{probexample}, we show that there exist balanced bipartite graphs on $2n$ vertices with $\delta_U+\delta_V\geq n+s^{s^{1/3}}$ which cannot be tiled with $K_{s,s}$.  An interesting problem would be to determine the largest possible value of $\delta_U+\delta_V$ such that $G[U,V]$ cannot be tiled with $K_{s,s}$.  We note that if $G[U,V]$ is a graph with $\delta_U+\delta_V\geq (1+\ep)n$, then $\delta_U\geq \ep n$ and thus we can apply Theorem \ref{main 1} or Theorem \ref{main 2} to obtain a tiling of $G$.

Finally, while we don't address the case of tiling with $K_{s,t}$ here, we point out that it is easy to prove an analog of Theorem \ref{main 2} for $K_{s,t}$.  In fact, even if we only assume $\delta_U+\delta_V\geq n$, we can tile $G$ with $K_{s,t}$:  the proof of Theorem \ref{main 2} is easy when there exists $\ell$ such that $|U_1|\leq \ell s$ and $|V_0\cup V_1|\geq \ell s$ by Claim \ref{V1>U1}, so we just remove copies of $K_{s,t}$ from $G[U_1, V_1]$, each with $t$ vertices in $U_1$, until the desired property holds and then we can finish the tiling as we do here.


\begin{thebibliography}{00}




\bibitem{CD} A. Czygrinow and L. DeBiasio, A note on bipartite graph tiling, \emph{SIAM J. Discrete Math.} \textbf{25}, no. 4, (2011), 1477--1489.


\bibitem{CDK} A. Czygrinow, L. DeBiasio, H.A. Kierstead, $2$-factors of bipartite graphs with asymmetric minimum degrees, \emph{SIAM J. Discrete Math.} \textbf{24} (2010), no. 2, 486--504. 


\bibitem{LouisThesis} L. DeBiasio, Optimal Degree Conditions for Spanning Subgraphs,  Ph.D. Thesis, Arizona State University, 2011.

\bibitem{HSz} A. Hajnal and E. Szemer\'edi, Proof of a conjecture of P. Erd\H{o}s, \emph{Combinatorial
Theory and its Application} ( P. Erd\H{o}s, A. R\'enyi, and V. T. S\'os, Eds.) North-Holland, London (1970) pp. 601--623.

\bibitem{HS} J. Hladk\'y and M. Schacht, Note on bipartite graph tilings, \emph{SIAM J. Discrete Math.} \textbf{24}, no. 2, (2010) 357--362.

\bibitem{KK} H.A. Kierstead and A.V. Kostochka, An Ore-type theorem on equitable coloring, \emph{J. Combin. Theory Ser. B} \textbf{98},  no. 1 (2008), 226--234.

\bibitem{KSSbu}  J. Koml\'{o}s, G. N. S\'{a}rk\"{o}zy and E. Szemer\'{e}di, Blow-up
lemma, \emph{Combinatorica} \textbf{17}, no. 1 (1997), 109--123.

\bibitem{Sz}  E. Szemer\'{e}di, Regular partitions of graphs, \emph{
Colloques Internationaux C.N.R.S., Problemes Combinatories et Theorie des
Graphes} (1978), 399--402.


\bibitem{W2} H. Wang, Bipartite graphs containing every possible pair of cycles, \emph{Discrete Mathematics} \textbf{207} (1999), 233--242.

\bibitem{Z} Y. Zhao, Bipartite graph tiling, \emph{SIAM J. Discrete Math.} \textbf{23}, no. 2 (2009), 888--900.

\end{thebibliography}
\end{document}